\documentclass[11pt]{amsart}
\usepackage{amssymb,amsthm,amsmath,epsfig,latexsym}
\usepackage{calc,times,verbatim}
\usepackage{geometry}
\usepackage{xcolor}
\usepackage{tikz}

\begin{document}
	
	\newcommand{\mmbox}[1]{\mbox{${#1}$}}
	\newcommand{\proj}[1]{\mmbox{{\mathbb P}^{#1}}}
	\newcommand{\Cr}{C^r(\Delta)}
	\newcommand{\CR}{C^r(\hat\Delta)}
	\newcommand{\affine}[1]{\mmbox{{\mathbb A}^{#1}}}
	\newcommand{\Ann}[1]{\mmbox{{\rm Ann}({#1})}}
	\newcommand{\caps}[3]{\mmbox{{#1}_{#2} \cap \ldots \cap {#1}_{#3}}}
	\newcommand{\Proj}{{\mathbb P}}
	\newcommand{\N}{{\mathbb N}}
	\newcommand{\Z}{{\mathbb Z}}
	\newcommand{\R}{{\mathbb R}}
	\newcommand{\A}{{\mathcal{A}}}
	\newcommand{\Tor}{\mathop{\rm Tor}\nolimits}
	\newcommand{\Ext}{\mathop{\rm Ext}\nolimits}
	\newcommand{\Hom}{\mathop{\rm Hom}\nolimits}
	\newcommand{\im}{\mathop{\rm Im}\nolimits}
	\newcommand{\rank}{\mathop{\rm rank}\nolimits}
	\newcommand{\supp}{\mathop{\rm supp}\nolimits}
	\newcommand{\arrow}[1]{\stackrel{#1}{\longrightarrow}}
	\newcommand{\CB}{Cayley-Bacharach}
	\newcommand{\coker}{\mathop{\rm coker}\nolimits}
	\newcommand{\fm}{\mathfrak m}
	\newcommand{\rees}{{R[It]}}
	\newcommand{\fiber}{{\mathcal{F}}(I)}
	\sloppy
	\theoremstyle{plain}
	
	\newtheorem*{thm*}{Theorem}
	\newtheorem{defn0}{Definition}[section]
	\newtheorem{prop0}[defn0]{Proposition}
	\newtheorem{quest0}[defn0]{Question}
	\newtheorem{thm0}[defn0]{Theorem}
	\newtheorem{lem0}[defn0]{Lemma}
	\newtheorem{corollary0}[defn0]{Corollary}
	\newtheorem{example0}[defn0]{Example}
	\newtheorem{remark0}[defn0]{Remark}
	\newtheorem{conj0}[defn0]{Conjecture}
	
	\newenvironment{defn}{\begin{defn0}}{\end{defn0}}
	\newenvironment{conj}{\begin{conj0}}{\end{conj0}}
	\newenvironment{prop}{\begin{prop0}}{\end{prop0}}
	\newenvironment{quest}{\begin{quest0}}{\end{quest0}}
	\newenvironment{thm}{\begin{thm0}}{\end{thm0}}
	\newenvironment{lem}{\begin{lem0}}{\end{lem0}}
	\newenvironment{cor}{\begin{corollary0}}{\end{corollary0}}
	\newenvironment{exm}{\begin{example0}\rm}{\end{example0}}
	\newenvironment{rem}{\begin{remark0}\rm}{\end{remark0}}
	
	\newcommand{\defref}[1]{Definition~\ref{#1}}
	\newcommand{\conjref}[1]{Conjecture~\ref{#1}}
	\newcommand{\propref}[1]{Proposition~\ref{#1}}
	\newcommand{\thmref}[1]{Theorem~\ref{#1}}
	\newcommand{\lemref}[1]{Lemma~\ref{#1}}
	\newcommand{\corref}[1]{Corollary~\ref{#1}}
	\newcommand{\exref}[1]{Example~\ref{#1}}
	\newcommand{\secref}[1]{Section~\ref{#1}}
	\newcommand{\remref}[1]{Remark~\ref{#1}}
	\newcommand{\questref}[1]{Question~\ref{#1}}
	
	\newcommand{\std}{Gr\"{o}bner}
	\newcommand{\jq}{J_{Q}}
	
	\title{Logarithmic derivations associated to line arrangements}
	\author{Ricardo Burity and \c{S}tefan O. Toh\v{a}neanu}

	 \keywords{hyperplane arrangements, logarithmic derivation, syzygy. \\
		\indent The first author was partially supported by CAPES - Brazil (grant: PVEX - 88881.336678/2019-01). \\
		\indent Burity's address: Departamento de Matem\'atica, Universidade Federal da Paraiba, J. Pessoa, Para\'iba, 58051-900, Brazil,
		Email: ricardo@mat.ufpb.br.\\
		\indent Tohaneanu's address: Department of Mathematics, University of Idaho, Moscow, Idaho 83844-1103, USA, Email: tohaneanu@uidaho.edu.}

	\begin{abstract}
		In this paper we give full classification of rank 3 line arrangements in $\mathbb P^2$ (over a field of characteristic 0) that have a minimal logarithmic derivation of degree 3. The classification presents their defining polynomials, up to a change of variables, with their corresponding affine pictures. We also analyze the shape of such a logarithmic derivation, towards obtaining criteria for a line arrangement to possess a cubic minimal logarithmic derivation.
	\end{abstract}
	
	\maketitle

\section{Introduction}

Let $\mathcal{A}=\{H_1,\ldots,H_s\}$ be a central hyperplane arrangement in $V=\mathbb{K}^n,$ where $\mathbb{K}$ is a field of characteristic zero. If $x_1,\ldots,x_n$ is a basis for $V^*$, then $H_i=V(l_i)$, $i=1,\ldots,s,$ for some linear forms $l_i$ in $S:=\mathrm{Sym}(V^*)=\mathbb{K}[x_1,\ldots,x_n].$

The module of the derivations of $S$ has the structure $\mathrm{Der}(S)=\oplus_{i=1}^{n}S\frac{\partial}{\partial{x_i}}$, that is, it is free of rank $n$. A \emph{logarithmic derivation} of $\mathcal{A}$ is an element $\theta \in \mathrm{Der}(S)$, such that $\theta(l_i) \in \langle l_i \rangle$, for all $i= 1,\ldots, s.$ The set of the logarithmic derivations forms an $S$-module, usually detonated by $\mathrm{Derlog}(\mathcal{A})$; the standard grading of $S$ induces a graded module structure on this $S$-module. Therefore, by the structure of the module of the derivations we can write $\theta = \sum_{i=1}^{n}P_i\frac{\partial}{\partial{x_i}}$ with $P_i$ homogeneous polynomials of the same degree, so we can define $\mathrm{deg}(\theta)=\mathrm{deg}(P_i)$.

Consider $F:=\Pi_{i=1}^{s}l_i$ the defining polynomial of $\mathcal{A}$. Since $F$ is a homogeneous polynomial the structure of module of the logarithmic derivations is well known, $$\mathrm{Derlog}(\mathcal{A})= \mathrm{Syz}(J_F)\oplus S\theta_E,$$ where $\mathrm{Syz}(J_F)$ is the first module of the syzygies on the Jacobian ideal of $F$, that is, the ideal of $S$ generated by the partial derivatives of the $F$ and $\theta_E=\sum_{i=1}^{n}x_i\frac{\partial}{\partial{x_i}}$ is the Euler derivation. In general, when $\mathrm{Syz}(J_F)$ is free, $V(F)$ is called a \emph{free divisor} (and therefore the hyperplane arrangement $\mathcal A$ is called {\em free}).

Maybe one of the most important conjecture in the field of hyperplane arrangement is Terao's Conjecture which states that if two hyperplane arrangements have isomorphic intersection lattices, then if one is free, then so is the other. So a deeper understanding of the generators of $\mathrm{Syz}(J_F)$ is of utmost importance to tackle this conjecture: one would like to describe these generators, or their degrees, only from the combinatorics of $\mathcal A$, if possible. In the past 5-6 years, there has been a lot of work on the {\em minimal degree of a Jacobian relation}, which is the minimal degree of a syzygy of $J_F$ (and here $V(F)$ can be any divisor, not necessarily a hyperplane arrangement):

$$r(\mathcal A)=mdr(\mathcal A):= \min_{r\in\mathbb Z}\{r|(\mathrm{Syz}(J_F))_r\neq 0\}.$$

Finding degrees of syzygies (and more generally, the shapes of the graded minimal free resolutions) has been one of the most fundamental topics in commutative/homological algebra; and the literature about this is extensive, being impossible to list it all here without leaving out some of it. Same is true also if one is interested in syzygies of Jacobian ideals of divisors (and therefore analyzing their singular locus), with a particular focus on plane projective divisors. Specifically to line arrangements in $\mathbb P^2$ (i.e., central rank 3 hyperplane arrangements in $\mathbb K^3$), we must mention the work of Dimca, Dimca-Sticlaru, and their coauthors: most of the time starting from results in \cite{duWa} where the invariant {\em mdr} seems to show up for the first time in the context of our project, they find connections between this invariant and the Tjurina number of a projective plane divisor (i.e., the degree of the Jacobian ideal), which in turn give restrictions on constructing line arrangements with prescribed combinatorics: \cite{TaDiSti,Di, DiSti, DiSti2}, and the citations therein.

Our goal is the following: in the spirit of \cite{To}, we will classify all line arrangements in $\mathbb P^2$ with $r(\mathcal A)=3$, also analyzing the shape of the corresponding cubic logarithmic derivation towards finding some structures in this shape. For the theoretical results we found essential the addition-deletion results for the minimal degree of logarithmic derivations of arrangements in \cite{TaDiSti}. As an appendix, at the end we briefly look at this invariant for graphic arrangements.

Below we summarize the results in \cite{TaDiSti}; for more definitions and notations clarifications, we refer the reader to \cite{OrTe}. Let $\mathcal A$ be a finite collection of (linear) hyperplanes in $\mathbb K^{\ell}$ with $\cap_{H\in\mathcal A}H=\{0\}$ (i.e., we say that $\mathcal A$ is central essential hyperplane arrangement of rank $\ell$). For $H\in\mathcal A$, let $$\mathcal A''=\mathcal A^{H}:=\{H\cap L|L\in\mathcal A\setminus\{H\}\}$$ be the {\em restriction}, and let $$\mathcal A':=\mathcal A\setminus\{H\}$$ be the {\em deletion}. Also, denote
$$r:=r(\mathcal A),\, r':=r(\mathcal A'), \, r'':=r(\mathcal A'').$$

\begin{thm} \label{additiondeletion} Let $\ell\geq 2$. With the above notations we have:
\begin{itemize}
  \item[(1)] \cite[Proposition 2.12]{TaDiSti}:
  \begin{itemize}
    \item[(1a)] $r'\leq r\leq r'+1$.
    \item[(1b)] If $|\mathcal A|-|\mathcal A''|>r$, then $r'=r$.
    \item[(1c)] If $|\mathcal A'|-|\mathcal A''|>r'$, then $r=r'$.
  \end{itemize}
  \item[(2)] \cite[Theorem 2.14]{TaDiSti}: If $r'<r''$, then $r=r'+1$.
  \item[(3)] \cite[Proposition 2.15]{TaDiSti}: If $r=r'$, then $r''\leq r$.
\end{itemize}
\end{thm}

Observe that the statement (3) is the counterpositive of (2), combined with (1a), so basically they are the same. But the flavors are different: (2) is the ``addition-deletion'' part, meaning that if we have information about $r'$ and $r''$, then we will obtain some information about $r$; whereas (3) is the ``restriction'' part, meaning that if we have information about $r$ and $r'$, then we will obtain information about $r''$.

\begin{rem}\label{upperbound} Let $\mathcal A\subset\mathbb P^{\ell-1}$ be a hyperplane arrangement of rank $\ell$. Let $X$ be a coatom, meaning that $X$ is a flat of rank $\ell-1$ in the lattice of intersection of $\mathcal A$. The {\em closure of $X$} is $cl(X):=\{H\in\mathcal A| X\in H\}$, i.e., the set of hyperplanes of $\mathcal A$ that contain $X$. The {\em multiplicity of $X$} is $\nu(X):=|cl(X)|$, and define $M=M(\mathcal A):=\max\{\nu(X)|X \mbox{ coatom}\}$.

Suppose $|\mathcal A|=s$ and suppose $M\leq s-2$. Then, using the same (classical) trick as in the proof of \cite[Theorem 1.2]{Di}, we have $$r(\mathcal A)\leq s-M.$$

Here is how: let $X$ be a coatom with $\nu(X)=M$. We can suppose $X=[0,\ldots,0,1]$. Then, the defining polynomial of $\mathcal A$ is $F=gh$, where $g,h\in S:=\mathbb K[x_1,\ldots,x_{\ell}]$ are products of the defining linear forms of the hyperplanes of $\mathcal A$, with $g(X)=0$ and $h(X)\neq 0$; and so $\deg(g)=M$ and $\deg(h)=s-M$.

For any homogeneous polynomial $P\in S$, denote $\displaystyle P_{x_i}:=\frac{\partial P}{\partial x_i}$. Then, $F_{x_{\ell}}=g\cdot h_{x_{\ell}}$. Together with Euler's formula $sF=x_1F_{x_{1}}+\cdots+x_{\ell-1}F_{x_{\ell-1}}+x_{\ell}F_{x_{\ell}}$, we get the Jacobian relation

$$(x_1h_{x_{\ell}})F_{x_{1}}+\cdots+(x_{\ell-1}h_{x_{\ell}})F_{x_{\ell-1}}+(x_{\ell}h_{x_{\ell}}-sh)F_{x_{\ell}}=0.$$ The only issue that can occur here is that the logarithmic derivation corresponding to this syzygy may be a multiple of Euler's derivation. If that were the case, since $\deg(h)=s-M\geq 2$, then there should exist a linear factor dividing both $h_{x_{\ell}}$ and $h$; an obvious contradiction.

\end{rem}

\begin{rem}\label{lowerbound} Same as above, let $\mathcal A\subset\mathbb P^{\ell-1}$ be a hyperplane arrangement of rank $\ell$. Define $0-Sing(\mathcal A):=\{X|X\mbox{ coatom}\}$ to be the set of coatoms of $\mathcal A$. This is a (reduced) set of points in $\mathbb P^{\ell-1}$. We define $\alpha_0(\mathcal A)$ to be the minimal degree of a hypersurface containing $0-Sing(\mathcal A)$. Then $$\alpha_0(\mathcal A)-1\leq r(\mathcal A).$$

To see this, we use the same trick as in \cite[Proposition 3.6]{To0}. Let $\displaystyle\theta=P_1\frac{\partial}{\partial x_1}+\cdots+P_{\ell} \frac{\partial}{\partial x_{\ell}}$ be a logarithmic derivation of degree $r(\mathcal A)$, that is not a multiple of the Euler derivation.

Let $X$ be a coatom, and suppose $X=H_1\cap\cdots\cap H_{\ell-1}$, for some $H_i\in\mathcal A$. So $X=[a_1,\ldots,a_\ell]$, and if $H_i=V(l_i),i=1,\ldots,\ell-1$, then $l_i(X)=0$, for all $i$. Suppose $l_i=c_i^1x_1+\cdots+c_i^{\ell}x_{\ell}, c_i^j\in \mathbb K$. Then, for each $i=1,\ldots,\ell-1$, $\theta(l_i)\in\langle l_i\rangle$, which evaluated at $X$ gives:

$$c_i^1P_1(X)+\cdots+c_i^{\ell}P_{\ell}(X)=0.$$

Either $P_j(X)=0$, for all $j=1,\ldots,\ell$, or the point $[P_1(X),\ldots,P_{\ell}(X)]\in V(l_1,\ldots,l_{\ell-1})$. In the later case, we must have that $[P_1(X),\ldots,P_{\ell}(X)]=X$ as projective points, hence $$P_j(X)=qa_j, j=1,\ldots,\ell, q\in\mathbb K\setminus\{0\}.$$ This means that for any $1\leq j<k\leq \ell$, $(x_kP_j-x_jP_k)(X)=0$. Since $x_kP_j-x_jP_k$ cannot be the zero polynomial for all $j<k$ (since $\theta$ is not a multiple of $\theta_E$), put together the two cases we have that all coatoms $X$ belong to a hypersurface of degree $\leq r(\mathcal A)+1$.
\end{rem}

\section{Line arrangements with cubic minimal log derivations}

Let us put the general concepts in our perspective. Let $\mathcal{A}=\{V(l_1),\ldots,V(l_s)\}\subset \mathbb{P}^2$ be a line arrangement of rank 3. So $V=\mathbb K^3$, and if $x,y,z$ is a basis for $V^*$, $l_1,\ldots,l_s$ are linear forms in $S:={\rm Sym}(V^*)=\mathbb K[x,y,z]$, with three of them being linearly independent. Also let $F=l_1\cdots l_s$ be the defining polynomial of $\mathcal A$.

{\em The singular locus} of $\mathcal A$, denoted $Sing(\mathcal A)$, is the set of the intersection points of lines of $\mathcal A$. For $P\in Sing(\mathcal A)$, {\em the multiplicity} of $P$, denoted $m_P$, is the number of lines of $\mathcal A$ intersecting at $P$. Combinatorially, $Sing(\mathcal A)$ is the set of rank 2 flats in the intersection lattice of $\mathcal A$, and $m_P=\nu(P)=\mu(P)+1$, where $\mu(P)$ is the value of the M\"{o}bius function at $P$. Also, denote $$m=m(\mathcal A):=\max\{m_P|P\in Sing(\mathcal A)\}.$$ Note that for a line arrangement in $\mathbb P^2$, $m(\mathcal A)$ and $M(\mathcal A)$ defined in the previous section are the same.

There are two very important formulas (with immediate inductive proofs on $s\geq 2$):

$$(\star) \hspace{2cm} \mbox{If } H\in \mathcal A,\mbox{ then }\sum_{P\in Sing(\mathcal A)\cap H}(m_P-1)=s-1,$$ and $$(\star\star) \hspace{2.5cm}\sum_{P\in Sing(\mathcal A)}{{m_P}\choose{2}}={{s}\choose{2}}.\hspace{2.7cm}$$

For $H\in\mathcal A$, we will denote $$|H|:=|\mathcal A''|=|Sing(\mathcal A)\cap H|,$$ the number of singularities of $\mathcal A$ lying on the line $H$. With this in mind we have the following crucial result:

\begin{thm} \label{lines} Let $\mathcal A$ be a line arrangement in $\mathbb P^2$ of rank 3, and let $H\in\mathcal A$. Then, using the notations above, we have:
\begin{itemize}
  \item[(i)] \cite[Theorem 3.3]{TaDiSti}: If $|H|\geq r'+2$, then $r=r'+1$.
  \item[(ii)] \cite[Theorem 3.4]{TaDiSti}: If $|H|\geq r+2$, then $r'=r-1$.
\end{itemize}
\end{thm}

The above theorem is a corollary to Theorem \ref{additiondeletion}: its proof uses the fact that $r''=|H|-1$. Based on the same fact we have

\begin{cor}\label{cor_lower} Let $\mathcal A\subset \mathbb P^2$ be a line arrangement of rank 3, with $|\mathcal A|=s$. Let $r:=r(\mathcal A)$ and $m:=m(\mathcal A)$. Then
$$r=s-m \mbox{ or } r\geq \min\{|\mathcal A^H|\, |\, H\in\mathcal A\}-1.$$
\end{cor}
\begin{proof} Denote $t:=\min\{|\mathcal A^H|\, |\, H\in\mathcal A\}$.

We prove the result by induction on $s\geq 3$.

If $s=3$, then $t=2$, $m=2$, and $r=1$, so the claim is true.

Suppose $s\geq 4$. From Remark \ref{upperbound}, we have $r\leq s-m$. Suppose $r\leq s-m-1$. We want to show $r\geq t-1$.

Let $H\in \mathcal A$ with $|\mathcal A^H|=t$. As before, $r'=r(\mathcal A\setminus\{H\})$, and $r''=r(\mathcal A^H)$; so $r''=t-1$. If rank of $\mathcal A'$ is 2, then $t=2$, and obviously $r\geq 2-1=1$.

Suppose rank of $\mathcal A'$ is 3, and suppose to the contrary that $r\leq t-2=r''-1$. Since $r'\leq r\leq r''-1$, by Theorem \ref{additiondeletion} (2), we get $r=r'+1$.

By induction, if $m':=m(\mathcal A')$, we have $r'=(s-1)-m'$, or $r'\geq |(\mathcal A')^L|-1$, for all $L\in\mathcal A'$.

If $r'=(s-1)-m'$, then $m'=m$ (the other option is $m'=m-1$, which contradicts $r'\leq r\leq s-m-1$). So $r'=r=s-m-1$; also a contradiction. So $r'\leq s-m-2$, and $r'\geq |(\mathcal A')^L|-1$, for all $L\in\mathcal A'$. So $r\geq |(\mathcal A')^L|$, for all $L\in\mathcal A'$.

If $H$ and $L$ intersect in a simple (i.e., double) point, then $|(\mathcal A')^L|=|\mathcal A^L|-1$; otherwise $|(\mathcal A')^L|=|\mathcal A^L|$. So, none the less we obtain $r\geq |(\mathcal A)^L|-1\geq t-1$. A contradiction.
\end{proof}

From this corollary we can conclude that $$\min\{|H|\, |\, H\in\mathcal A\}-1\leq r\leq s-m.$$ In the perspective of Remark \ref{lowerbound}, we ask what is the relation between $\alpha_0(\mathcal A)$ and $\min\{|H|\,|\, H\in\mathcal A\}$, for an arbitrary line arrangement $\mathcal A\subset \mathbb{P}^2$ of rank 3.

\medskip

\begin{rem} \label{recursion_remark} Suppose we want to classify all line arrangements $\mathcal A\subset\mathbb P^2$ with $r(\mathcal A)=d$ for some $d\geq 2$. Let $H\in\mathcal A$, and let $\mathcal A'=\mathcal A\setminus\{H\}$. By Theorem \ref{additiondeletion} (1a), $r'=d-1$, or $r'=d$. Then, by Theorem \ref{lines} (i) and (ii), we have either:
\begin{itemize}
  \item[(a)] $\mathcal A$ is obtained recursively from the classification of all line arrangements $\mathcal B\subset\mathbb P^2$ with $r(\mathcal B)=d-1$ by adding a line $H$ that intersects all other lines of $\mathcal B$ in $\geq (d-1)+2=d+1$ points, or
  \item[(b)] every line of $\mathcal A$ has $\leq d+1$ points on it.
\end{itemize}
\end{rem}

\medskip

Case (b) imposes some restrictions on the possible types of line arrangements. Let $P\in Sing(\mathcal A)$, with $m_P=m=m(\mathcal A)$. Since the rank of $\mathcal A$ is 3, there must exist a line $H\in\mathcal A$ that doesn't pass through $P$. So $$m\leq |H|\leq d+1,$$ and this is true for any line that doesn't pass through $P$.

If furthermore, $m=d+1$, then $\mathcal A$ is a supersolvable arrangement with modular point $P$. Hence it is free with exponents ${\rm exp}(\mathcal A)=(1,d, s-(d+1))$, where $|\mathcal A|=s$.

\medskip

\begin{exm} \label{example1} If $\mathcal A\subset\mathbb P^2$ is a line arrangement of rank 3, then Remarks \ref{upperbound} and \ref{lowerbound} give $$\alpha_0(\mathcal A)-1\leq r(\mathcal A)\leq |\mathcal A|-m(\mathcal A).$$ Let us look at Ziegler (\cite{Zi}) - Yuzvinsky (\cite{Yu}) example:

$$\mathcal A_1 = V(xyz(x+y+z)(2x+y+z)(2x+3y+z)(2x+3y+4z)(3x+5z)(3x+4y+5z))$$
$$\mathcal A_2 = V(xyz(x+y+z)(2x+y+z)(2x+3y+z)(2x+3y+4z)(x+3z)(x+2y+3z)).$$ These two arrangements have the same combinatorics, yet $r_1:=r(\mathcal A_1)=6$, and $r_2:=r(\mathcal A_2)=5$. Also for both arrangements $|\mathcal A_i|-m(\mathcal A_i)=9-3=6$, and $\alpha_0(\mathcal A_i)-1=6-1=5$. So $r_1$ achieves the upper bound, and $r_2$ achieves the lower bound.

Also, in both examples $|H|=6\leq r_i+1$, for any $H\in\mathcal A_i$. Therefore both options in the statement of Corollary \ref{cor_lower} are satisfied.

\end{exm}

\subsection{The case of $r=2$.} Remark \ref{recursion_remark} can shortcut quite a bit the calculations done in \cite{To} in order to obtain the classification of rank 3 line arrangements with a quadratic minimal logarithmic derivation. In the above discussion, $d=2$.

For case (a), either by using \cite[Proposition 4.29]{OrTe}, or \cite[Section 2.1]{To}, the line arrangement $\mathcal B$ with $r(\mathcal B)=d-1=1$ is a pencil of $\geq d+1=3$ lines, and an extra line not part of the pencil; or a triangle of 3 lines. There are only two different ways one can add another line to $\mathcal B$ to obtain $\mathcal A$, with at least 3 points on it, and these give \cite[Theorem 2.4 parts (1) and (2)]{To}.

For case (b), $m$ can be only 2, or 3. If $m=2$, then $\mathcal A$ is the generic line arrangement of $s$ lines, and since in this case $Sing(\mathcal A)$ is a star configuration, its defining ideal is generated in degree $s-1$ (see \cite{GeHaMi}).  From Remark \ref{lowerbound} we require $s-1\leq d+1=3$, then $s\leq 4$, leading to $s=4$ which is a special case of \cite[Theorem 2.4 (2)]{To}. If $m=3$, then $\mathcal A$ is a supersolvable line arrangement with exponents $(1,2,s-3)$. At the same time, by \cite[Theorem 3.2 (1)]{DiSti}, we have $2\geq (s-2)/2$, leading to $s\leq 6$. Therefore $\mathcal A$ is the supersolvable arrangement with exponents $(1,2,3)$ which is exactly \cite[Theorem 2.4 (3)]{To}, or $(1,2,2)$ which is a special case of \cite[Theorem 2.4 (1)]{To}.

In summary this is the main result of \cite{To}; parts (1), (2), (3), correspond to (I), (II), and respectively, (III).

\begin{thm}[\cite{To} Theorem 2.4]\label{classification}
If $\mathcal{A}$ has a minimal quadratic syzygy on its Jacobian ideal, but not a linear syzygy, then, up to a change of coordinates, $\mathcal{A}$ is one of the following three types of arrangements with defining polynomials:

\begin{itemize}
	\item [(I)] $F=xyz(x+y)\Pi_{j=4}^{s}(t_jy+z),~t_j\neq 0.$
	\item [(II)] $F=xyz(x+y+z)\Pi_{j=4}^{s}(t_jy+z),~t_j\neq 0,1.$
	\item [(III)] $F=xyz(x+y+z)(x+z)(y+z).$
\end{itemize}

See their affine pictures below:

\medskip
$$\stackrel
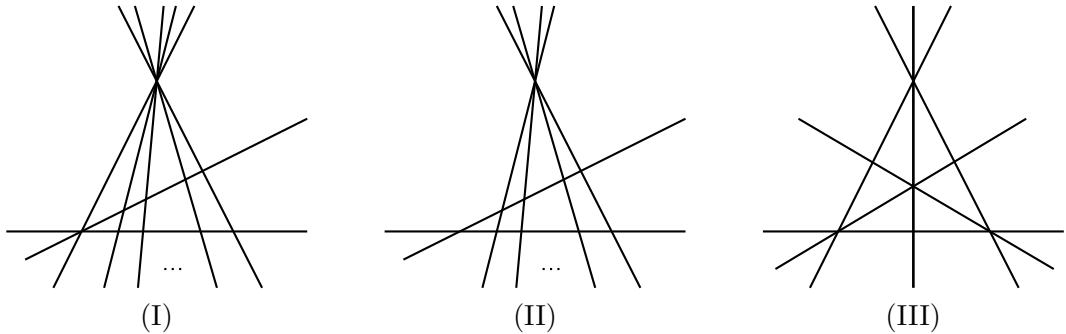
{
\begin{tikzpicture}
\draw[thick] (-1, 0) -- (3, 0);
\draw[thick] (-0.75, -3/8) -- (3, 1.5);
\draw[thick] (-0.75/2, -0.75) -- (1.5, 3);
\draw[thick] (0.3, -0.75) -- (345/275, 3);
\draw[thick] (0.75, -0.75) -- (12/11, 3);
\draw[thick] (1.8, -0.75) -- (24.375/34.375, 3);
\draw[thick] (2.4, -0.75) -- (135/275, 3);
\draw[thick,dotted] (1.1, -0.5) -- (1.35, -0.5);
\end{tikzpicture}}{\stackrel{}{(\mathrm{I})}}
\hspace{1cm}\stackrel{
	\begin{tikzpicture}
	\draw[thick] (-1, 0) -- (3, 0);
	\draw[thick] (-0.75, -3/8) -- (3, 1.5);
	\draw[thick] (0.3, -0.75) -- (345/275, 3);
	\draw[thick] (0.75, -0.75) -- (12/11, 3);
	\draw[thick] (1.8, -0.75) -- (24.375/34.375, 3);
	\draw[thick] (2.4, -0.75) -- (135/275, 3);
	\draw[thick,dotted] (1.1, -0.5) -- (1.35, -0.5);
	\end{tikzpicture}}{\stackrel{}{(\mathrm{II})}}
\hspace{1cm}\stackrel{
	\begin{tikzpicture}
	\draw[thick] (-1, 0) -- (3, 0);
	\draw[thick] (-0.75/2, -0.75) -- (1.5, 3);
	\draw[line width=1pt] (1, -0.75) -- (1, 3);
	\draw[thick] (2.4, -0.75) -- (135/275, 3);
    \draw[thick] (-5/6, -0.5) -- (2.5, 1.5);
    \draw[thick] (43/15, -0.5) -- (-29/55, 1.5);
	\end{tikzpicture}}{\stackrel{}{(\mathrm{III})}}$$

\end{thm}

\medskip

\subsection{The case of $r=3$.} As in the previous situation we apply Remark \ref{recursion_remark}, but when $d=3$. Then the classification is the following:

\begin{thm} \label{cubic} Let $\mathcal A\subset \mathbb P^2$ be a line arrangement of rank 3, with $r(\mathcal A)=3$. Then, up to a change of coordinates, $\mathcal A$ has one of the following defining polynomials (and corresponding (affine) pictures):
	\begin{itemize}
\item[(Ia)] $ F=xyz(x+y)\textcolor{blue}{(bx+y)} \prod_{j=6}^s(t_jy+z), t_j\neq 0, s\geq 7.$

\item[(Ib)] $ F=xyz(x+y)\textcolor{blue}{(bx+z)}\prod_{j=6}^s(t_jy+z), t_j\neq 0, s\geq 6.$

\item[(Ic)] $ F=xyz(x+y)(y+z)\textcolor{blue}{(-x+z)}\prod_{j=7}^s(t_jy+z), t_j\neq 0,1, s\geq 7.$

\item[(Id)] $ F=xyz(x+y)\textcolor{blue}{(ax+by+z)}\prod_{j=6}^s(t_jy+z), t_j\neq 0, s\geq 5.$

\item[(IIa)]$ F=xyz(x+y+z)\textcolor{blue}{(ax+by+z)}\prod_{j=6}^s(t_jy+z), t_j\neq 0,~s\geq 5.$

\item[(IIb)]$ F=xyz(x+y+z)\textcolor{blue}{(bx+y)}\prod_{j=6}^s(t_jy+z), t_j\neq 0, s\geq 6.$

\item[(IIIa)]$ F=xyz(x+z)(y+z)(x+y+z)\textcolor{blue}{(ax+by+z)}.$

\item[(IIIb)]$ F=xyz(x+z)(y+z)(x+y+z)\textcolor{blue}{(ax+ay+z)}.$

\item[(IIIc)]$ F=xyz(x+z)(y+z)(x+y+z)\textcolor{blue}{(x+y+2z)}.$

\item[(IV)]$ F=xy(x-z)(y-z)(x-2z)(y-2z)(x-y).$

\item[(Va)]$ F=xyz(x^2-y^2)(x^2-z^2)(y^2-z^2).$

\item[(Vb)]$ F=xyz(x^2-z^2)(y^2-z^2)(x-y).$

\item[(Vc)]{\bf Updated December 13, 2023:} $ F=xy(x^2-y^2)(x^2-z^2)(y^2-z^2).$ This addition was mentioned to us by Wayne Ng Kwing King and Jean Vall\`{e}s. This is derived from (Va) after deleting the line at infinity $z=0$.
	\end{itemize}

The parameters $a$ and $b$ are any elements of $\mathbb K$ chosen such that we do not repeat any of the lines already selected, nor we obtain undesirable concurrencies; see the corresponding pictures.

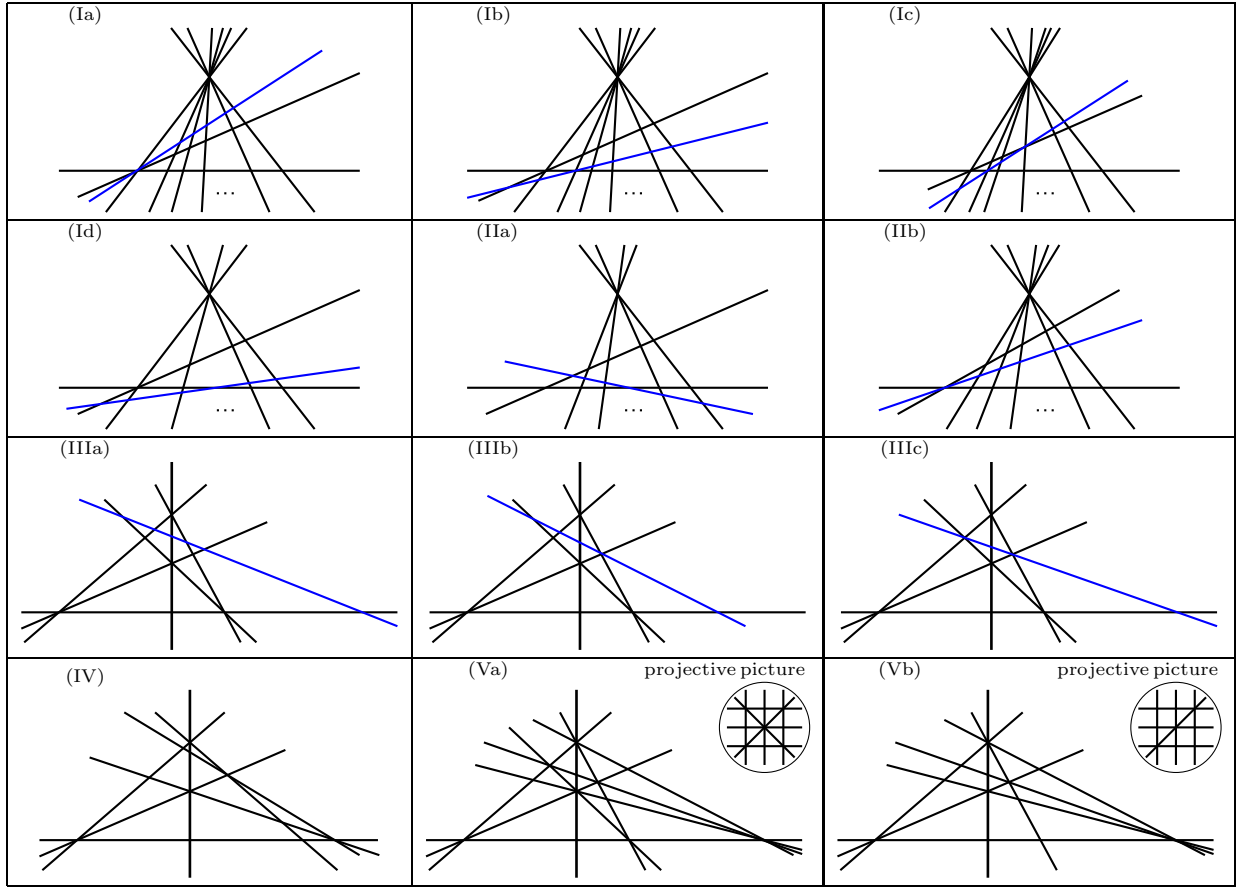
\begin{center}
	\begin{tiny}	
		\begin{tabular}{|c|c|c|}
			\hline $\stackrel{\displaystyle{\Large(\mathrm{Ia})}\hspace{3.3cm}}{
				\begin{tikzpicture}
				\draw[thick] (-1, -0.2) -- (3, -0.2);
				\draw[thick] (-0.75, -0.55) -- (3, 1.1);
				\draw[thick] (-0.75/2, -0.75) -- (1.5, 1.7);
				\draw[thick] (0.2, -0.75) -- (4.4375/3.4375, 1.7);
				\draw[thick] (0.5, -0.75) -- (13/11, 1.7);
				\draw[thick] (0.9, -0.75) -- (57/55, 1.7);
				\draw[thick] (1.8, -0.75) -- (24.375/34.375, 1.7);
				\draw[thick] (2.4, -0.75) -- (135/275, 1.7);
				\draw[blue, thick] (-0.6, -0.61) -- (2.5, 1.4);
				\draw[thick,dotted] (1.1, -0.5) -- (1.35, -0.5);
				\end{tikzpicture}}$  & $\stackrel{\displaystyle{\Large(\mathrm{Ib})}\hspace{3.3cm}}{
				\begin{tikzpicture}
				\draw[thick] (-1, -0.2) -- (3, -0.2);
				\draw[thick] (-0.85, -0.6) -- (3, 1.1);
				\draw[thick] (-0.75/2, -0.75) -- (1.5, 1.7);
				\draw[thick] (0.2, -0.75) -- (4.4375/3.4375, 1.7);
				\draw[thick] (0.5, -0.75) -- (13/11, 1.7);
				\draw[thick] (0.9, -0.75) -- (57/55, 1.7);
				\draw[thick] (1.8, -0.75) -- (24.375/34.375, 1.7);
				\draw[thick] (2.4, -0.75) -- (135/275, 1.7);
				\draw[blue, thick] (-1, -0.56) -- (3,0.44);
				\draw[thick,dotted] (1.1, -0.5) -- (1.35, -0.5);
				\end{tikzpicture}}$ & $\stackrel{\displaystyle{\Large(\mathrm{Ic})}\hspace{3.3cm}}{
				\begin{tikzpicture}
				\draw[thick] (-1, -0.2) -- (3, -0.2);
				\draw[thick] (-0.35, -0.45) -- (2.5, 0.8);
				\draw[thick] (-211/1700 , -0.75) -- (6586/4675, 1.7);
				\draw[thick] (0.2, -0.75) -- (4.4375/3.4375, 1.7);
				\draw[thick] (0.4, -0.75) -- (67/55, 1.7);
				\draw[thick] (0.9, -0.75) -- (57/55, 1.7);
				\draw[thick] (1.8, -0.75) -- (24.375/34.375, 1.7);
				\draw[thick] (2.4, -0.75) -- (135/275, 1.7);
				\draw[blue, thick] (-1/3, -0.7) -- (104/45,1);
				\draw[thick,dotted] (1.1, -0.5) -- (1.35, -0.5);
				\end{tikzpicture}}$ \\
				\hline  $\stackrel{\displaystyle{\Large(\mathrm{Id})}\hspace{3.3cm}}{
				\begin{tikzpicture}
				\draw[thick] (-1, -0.2) -- (3, -0.2);
				\draw[thick] (-0.75, -0.55) -- (3, 1.1);
				\draw[thick] (-0.75/2, -0.75) -- (1.5, 1.7);
				\draw[thick] (0.5, -0.75) -- (13/11, 1.7);
				\draw[thick] (1.8, -0.75) -- (24.375/34.375, 1.7);
				\draw[thick] (2.4, -0.75) -- (135/275, 1.7);
				\draw[blue, thick] (-0.9, -0.48) -- (3,0.07);
				\draw[thick,dotted] (1.1, -0.5) -- (1.35, -0.5);
				\end{tikzpicture}}$ & $\stackrel{\displaystyle{\Large(\mathrm{IIa})}\hspace{3.2cm}}{
				\begin{tikzpicture}
				\draw[thick] (-1, -0.2) -- (3, -0.2);
				\draw[thick] (-0.75, -0.55) -- (3, 1.1);
				\draw[thick] (0.3, -0.75) -- (345/275, 1.7);
				\draw[thick] (0.75, -0.75) -- (12/11, 1.7);
				\draw[thick] (1.8, -0.75) -- (24.375/34.375, 1.7);
				\draw[thick] (2.4, -0.75) -- (135/275, 1.7);
				\draw[blue, thick] (-0.5, 0.15) -- (2.8, -0.55);
				\draw[thick,dotted] (1.1, -0.5) -- (1.35, -0.5);
				\end{tikzpicture}}$  & $\stackrel{\displaystyle{\Large(\mathrm{IIb})}\hspace{3.3cm}}{
				\begin{tikzpicture}
				\draw[thick] (-1, -0.2) -- (3, -0.2);
				\draw[thick] (-0.75, -0.55) -- (2.2, 1.1);
				\draw[thick] (-0.1, -0.75) -- (7/5, 1.7);
				\draw[thick] (0.3, -0.75) -- (345/275, 1.7);
				\draw[thick] (0.75, -0.75) -- (12/11, 1.7);
				\draw[thick] (1.8, -0.75) -- (24.375/34.375, 1.7);
				\draw[thick] (2.4, -0.75) -- (135/275, 1.7);
				\draw[blue, thick] (-1, -0.5) -- (2.5, 0.7);
				\draw[thick,dotted] (1.1, -0.5) -- (1.35, -0.5);
				\end{tikzpicture}}$ \\
			\hline  $\stackrel{\displaystyle{\Large(\mathrm{IIIa})}\hspace{3.3cm}}{
				\begin{tikzpicture}
				\draw[thick] (-1, 0) -- (4, 0);
				\draw[thick] (-833/866,-0.4) -- (1267/866, 1.7);
				\draw[thick] (145.299/186.190, 1.7) -- (356.811/186.190, -0.4);
				\draw[line width=1pt] (1, 2) -- (1, -0.5);
				\draw[thick] (-1,-13/60) -- (59/26, 1.2);
				\draw[thick] (497/4810, 1.5) -- (2045/962,-0.4);
				\draw[thick,blue] (-3/13,1.5) -- (4,-0.185);
				\end{tikzpicture}}$ & $\stackrel{\displaystyle{\Large(\mathrm{IIIb})}\hspace{3.3cm}}{
				\begin{tikzpicture}
				\draw[thick] (-1, 0) -- (4, 0);
				\draw[thick] (-833/866,-0.4) -- (1267/866, 1.7);
				\draw[thick] (145.299/186.190, 1.7) -- (356.811/186.190, -0.4);
				\draw[line width=1pt] (1, 2) -- (1, -0.5);
				\draw[thick] (-1,-13/60) -- (59/26, 1.2);
				\draw[thick] (497/4810, 1.5) -- (2045/962,-0.4);
				\draw[thick,blue] (-3/13,1.55) -- (3.2,-0.185);
				\end{tikzpicture}}$ & $\stackrel{\displaystyle{\Large(\mathrm{IIIc})}\hspace{3.3cm}}{
				\begin{tikzpicture}
				\draw[thick] (-1, 0) -- (4, 0);
				\draw[thick] (-833/866,-0.4) -- (1267/866, 1.7);
				\draw[thick] (145.299/186.190, 1.7) -- (356.811/186.190, -0.4);
				\draw[line width=1pt] (1, 2) -- (1, -0.5);
				\draw[thick] (-1,-13/60) -- (59/26, 1.2);
				\draw[thick] (497/4810, 1.5) -- (2045/962,-0.4);
				\draw[thick,blue] (-3/13,1.3) -- (4,-0.185);
				\end{tikzpicture}}$ \\		
			\hline
			$\stackrel{\displaystyle{\Large(\mathrm{IV})}\hspace{3.3cm}}{
				\begin{tikzpicture}
				\draw[thick] (-1, 0) -- (3.5, 0);
				\draw[thick] (-833/866,-0.4) -- (1267/866, 1.7);
				\draw[thick] (465/866, 1.7) -- (2565/866, -0.4);
				\draw[line width=1pt] (1, 2) -- (1, -0.5);
				\draw[thick] (-1,-13/60) -- (59/26, 1.2);
				\draw[thick] (-1/3, 1.1) -- (95/27, -0.2);
				\draw[thick] (115.987/911.898, 1.7) -- (1484.200/455.949, -0.2);
				\end{tikzpicture}}$ & $\stackrel{\displaystyle{\Large\hspace{0.5cm}(\mathrm{Va})}\hspace{1.8cm}\mathrm{projective}~\mathrm{picture}}{
				\begin{tikzpicture}
				\draw[thick] (-1, 0) -- (4, 0);
				\draw[thick] (-833/866,-0.4) -- (1267/866, 1.7);
				\draw[thick] (145.299/186.190, 1.7) -- (356.811/186.190, -0.4);
				\draw[line width=1pt] (1, 2) -- (1, -0.5);
				\draw[thick] (-1,-13/60) -- (59/26, 1.2);
				\draw[thick] (497/4810, 1.5) -- (2045/962,-0.4);
				\draw[thick] (1093/2598, 1.6) -- (10093/2598, -0.2);
				\draw[thick] (-9/26, 1) -- (4, -0.13);
				\draw[thick] (-3/13,1.3) -- (4,-0.185);
				\draw[thick] (3.25,1) -- (3.25,2);
				\draw[thick] (3.5,1) -- (3.5,2);
				\draw[thick] (3.75,1) -- (3.75,2);
				\draw[thick] (3,1.75) -- (4,1.75);
				\draw[thick] (3,1.5) -- (4,1.5);
				\draw[thick] (3,1.25) -- (4,1.25);
				\draw[thick] (3.1,1.1) -- (3.9,1.9);
				\draw[thick] (3.1,1.9) -- (3.9,1.1);
				\draw (3.5, 1.5) circle (.6);
				\end{tikzpicture}}$  & $\stackrel{\displaystyle{\Large\hspace{0.5cm}(\mathrm{Vb})}\hspace{1.8cm}\mathrm{projective}~\mathrm{picture}}{
				\begin{tikzpicture}
				\draw[thick] (-1, 0) -- (4, 0);
				\draw[thick] (-833/866,-0.4) -- (1267/866, 1.7);
				\draw[thick] (145.299/186.190, 1.7) -- (356.811/186.190, -0.4);
				\draw[line width=1pt] (1, 2) -- (1, -0.5);
				\draw[thick] (-1,-13/60) -- (59/26, 1.2);
				\draw[thick] (1093/2598, 1.6) -- (10093/2598, -0.2);
				\draw[thick] (-9/26, 1) -- (4, -0.13);
				\draw[thick] (-3/13,1.3) -- (4,-0.185);
				\draw[thick] (3.25,1) -- (3.25,2);
				\draw[thick] (3.5,1) -- (3.5,2);
				\draw[thick] (3.75,1) -- (3.75,2);
				\draw[thick] (3,1.75) -- (4,1.75);
				\draw[thick] (3,1.5) -- (4,1.5);
				\draw[thick] (3,1.25) -- (4,1.25);
				\draw[thick] (3.1,1.1) -- (3.9,1.9);
				\draw (3.5, 1.5) circle (.6);
				\end{tikzpicture}}$ \\
			\hline
                 \end{tabular}
         	\end{tiny}

         \begin{tiny}
         	           \begin{tabular}{|c|}
           $\stackrel{\displaystyle{\Large\hspace{0.5cm}(\mathrm{Vc})}\hspace{1.8cm}\mathrm{projective}~\mathrm{picture}}{
            	\begin{tikzpicture}
            		\draw[thick] (-1, 0) -- (4, 0);
            		\draw[thick] (-833/866,-0.4) -- (1267/866, 1.7);
            		\draw[thick] (145.299/186.190, 1.7) -- (356.811/186.190, -0.4);
            		\draw[line width=1pt] (1, 2) -- (1, -0.5);
            		\draw[thick] (-1,-13/60) -- (59/26, 1.2);
            		\draw[thick] (497/4810, 1.5) -- (2045/962,-0.4);
            		\draw[thick] (-9/26, 1) -- (4, -0.13);
            		\draw[thick] (-3/13,1.3) -- (4,-0.185);
            		\draw[thick] (3.25,1) -- (3.25,2);
            		\draw[thick] (3.5,1) -- (3.5,2);
            		\draw[thick] (3.75,1) -- (3.75,2);
            		\draw[thick] (3,1.75) -- (4,1.75);
            		\draw[thick] (3,1.5) -- (4,1.5);
            		\draw[thick] (3,1.25) -- (4,1.25);
            		\draw[thick] (3.1,1.1) -- (3.9,1.9);
            		\draw[thick] (3.1,1.9) -- (3.9,1.1);
            	      \end{tikzpicture}}$\\
                     \hline
           \end{tabular}
            	\end{tiny}\hspace*{10.945cm}
                 \end{center}
\end{thm}

\medskip

\begin{proof} The cases (Ia)-(IIIc) are obtained by adding a line (depicted here in blue) with at least $d+1=4$ intersection points to a line arrangement $\mathcal B$, with $r(\mathcal B)=2$ (so a part of the classification in Theorem \ref{classification}).

The arrangements not obtained recursively are treated in case (b) of Remark \ref{recursion_remark}.

\medskip

If $m=2,3$, then $\mathcal A$ is an arrangement with only double and triple points (i.e., for any $P\in Sing(\mathcal A)$, $m_P=2,3$), and therefore, by \cite[Theorem 3.2 (1)]{DiSti}, we have $$r(\mathcal A)=3\geq\frac{s-2}{2},$$ so $s\leq 8$. Also, every line has at most $d+1=4$ points on it. Suppose $T_H$ is the number of triple points on a line $H\in \mathcal A$, and $D_H$ is the number of double points on the same line $H$. So $$D_H+T_H\leq 4.$$ From formula $(\star)$ we also have:
$$D_H+2T_H=s-1.$$ Of course, if $s\geq 6$, then $T_H\geq 1$ for all $H\in\mathcal A$, because otherwise (i.e., $T_H=0$) $D_H\geq 5$ which is impossible to happen for line $H$.

$\bullet$ If $T_H=1$, then $D_H\leq 3$, so $s\leq 3+2\cdot 1+1=6$.

$\bullet$ If $T_H=2$, then $D_H\leq 2$, so $s\leq 2+2\cdot 2 +1=7$.

$\bullet$ If $T_H=3$, then $D_H\leq 1$, and so $s\leq 1+2\cdot 3 +1=8$.

$\bullet$ If $T_H\geq 4$, then $s\geq 9$, contradiction.

With all of this information, we can apply \cite[Theorem 4.18]{TaDiSti}. Below, $n_2$ is the number of double points, and $n_3$ is the number of triple points. From $(\star\star)$ we have $\displaystyle n_2+3n_3={{s}\choose{2}}$.
\begin{itemize}
  \item[i.] $n_3=0$, meaning that $m=2$. Then $\mathcal A$ is a generic line arrangement of $s=3+2=5$ lines. This is the special case of (IIa) with $s=5$.
  \item[ii.] $1\leq n_3\leq 3$. Then $s=3+3=6$. When $n_3=1$, we have the special case (IIa) with $s=6$; when $n_3=2$, we have the special cases (Id) and (IIb) with $s=6$; when $n_3=3$, we have the special case (Ib) with $s=6$.
  \item[iii.] $n_3=4$. If $s=6$, then $n_2=15-12=3$. Since $D_H+2T_H=5$, we can have $T_H=2, D_H=1$, or $T_H=1, D_H=3=n_2$. This second case cannot happen, because otherwise, for all other lines $D_H=0$; contradiction. So $T_H=2$ and $D_H=1$ for all six lines $H$ of $\mathcal A$. This is the case of $\mathcal A(2,2,3):=V((x^2-y^2)(x^2-z^2)(y^2-z^2))$, which has $r=2$ (i.e., situation (III) in Theorem \ref{classification}). \\ Then \cite[Theorem 4.18 (3)]{TaDiSti} is saying that we must have $s=r+4=7$, and that $\mathcal A$ is obtained by adding a generic line $H$ to the line arrangement $\mathcal A(2,2,3)$ (i.e., situation (IIIa) in Theorem \ref{cubic}). But in this case $|H|=6$, a contradiction with our setup for case (b).
  \item[iv.] $n_3=5$. \cite[Theorem 4.18 (4)(A)]{TaDiSti} implies that $T_L=1$ which by the first bullet above, leads to $s=6$. So $D_L=5-2=3$, yet $n_2=15-15=0$. Contradiction. So $s=7$, and $\mathcal A$ is obtained by adding the line $L$ to one of the double points of $\mathcal A(2,2,3)$ (i.e., situation (IIIb) in Theorem \ref{cubic}). But this contradicts our case (b): $|L|=5\nleq 4$.  \\ So we are left with case (4)(B) of that theorem, with $s=7$: we are adding $7-7=0$ generic lines to the arrangement denoted there by $\mathcal B$; this is the same arrangement as $\mathcal A_4:=V(xy(x-z)(y-z)(x-2z)(y-2z)(x-y))$ in \cite[Section 4]{To1}. So a new type, denoted with (IV).
  \item[v.] $n_3=6$. In this instance $s=7, 8$, and $T_H=2,3$ for all $H$. Through any triple point there must be a line $H$ with $T_H=2$ (otherwise we will have $n_3\geq 7$). But from the second bullet above, we must have $s=7$. Hence, $n_2=21-18=3$, and $D_H=2$. If $L$ is a line with $T_L=3$, then $D_H=0$. Because $n_2=3$, every triple point will have one line with $T=2$ and two lines with $T=3$ through it. Therefore $\mathcal A$ is the non-Fano arrangement (which is free with exponents $(1,3,3)$); this is type (IIIc).
  \item[vi.] $n_3=7$. Then every line through a triple point will have $T\leq 3$ (see the forth bullet above). If $s=7$, then $n_2=0$, and such an arrangement is the Fano plane, which is realizable only over a field of characteristic 2. \\ So $s=8$, and from the third bullet above, for every $H\in \mathcal A$, we must have $D_H=1$, and $T_H=3$. Also, $n_2=28-21=7$. But this is impossible, since every double point belongs to exactly two lines, and every line has exactly one double point, hence $n_2=8/2=4$.
  \item[vii.] $n_3\geq 8$. Then, $s=8$, $n_2=28-24=4$ (here $n_3=8$), or $n_2=28-27=1$ (here $n_3=9$). Also, from the third bullet $T_H=3$, and $D_H=1$ for all $H\in\mathcal A$. Therefore $n_3=9$ cannot happen; just look at two lines passing through the same triple point. \\ So $n_3=8$. But there is only one such arrangement, namely the deleted Hasse arrangement: $\mathcal A=V((x^2+xy+y^2)(x^3-z^3)(y^3-z^3))$. Computations with Macaulay 2 (\cite{EiGrSt}) show that $r(\mathcal A)=4$, so contradiction.
\end{itemize}

\bigskip

Now suppose $m=3+1=4$. So $\mathcal A$ is supersolvable with exponents $(1,3,s-4)$. If $s-4=3$, then $s=7$, and we can have two special cases of types (Ia) and (Ic) in Theorem \ref{cubic}.

Suppose $s\geq 8$. By \cite[Theorem 1.1]{AbDi}, we have $s\leq 3m-3=9$.

From \cite[Corollary 2]{HaHa}, if there is another modular point of multiplicity $m'<m=4$, then $s=m+m'-1\leq 6$. So in order to proceed with our situation, we must consider the case of {\em $m-$homogeneous} supersolvable line arrangements (i.e., all modular points have the same multiplicity $m$). As it is worked out in \cite[Section 3]{HaHa}, and very well summarized in \cite[Theorem 1.2]{AbDi}, any 4-homogeneous supersolvable arrangement, with $n_4\geq 3$ is, after a change of variables, $\mathcal A(2,1,3):=V(xyz(x^2-y^2)(x^2-z^2)(y^2-z^2))$; this is type (Va) in our list.

If $n_4=2$, let $P,Q\in Sing(\mathcal A)$ be with $m_P=m_Q=m=4$. Let $H:=\overline{PQ}$. Then $D_H+2T_H+2\cdot (4-1)=s-1\leq 8$.

\begin{itemize}
  \item[i.] If $s=8$, then $D_H=1$ and $T_H=0$. If $L$ is the line that intersects $H$ at the double point, then $|L|=4$, and therefore $T_L=3$ and $D_L=1$. So $\mathcal A$ has defining polynomial $$xyz(x^2-z^2)(y^2-z^2)(x-y),$$ which is type (Vb) in our list.
  \item[ii.] If $s=9$, for real arrangements, as \cite[Section 3.2.3]{HaHa} shows, this is impossible. For complex arrangements, we proceed with the same ideas as above. We can have $D_H=2$ and $T_H=0$, or $D_H=0$ and $T_H=1$. In the first case, if $L$ is a line that intersects $H$ at one of the double points, then $|L|=D_L+T_L=4$, and $D_L+2T_L=9-1=8$. So $T_L=4$, giving that $D_L=0$, contradiction with the fact that $L$ has at least one double point. In the second case, any line $L\neq H$ through the one triple point $R$ has, by the same calculation we just did, $T_L=4$, and $D_L=0$. By \cite[Conjecture 3.3]{AnTo}, $\mathcal A$ is not realizable over the field of complex numbers.\footnote{The conjecture is shown for $s\leq 12$ in \cite[Corollary 3.5]{AnTo}, and to our overwhelming excitement, the conjecture has been proven in general by Takuro Abe in \cite{Ab}.} Therefore no other type is added to the list.
\end{itemize}

{\bf Updated December 13, 2023:} Suppose $n_4=1$, and therefore there is only one modular point, $P$, which has multiplicity $m_P=4$. Let $H\in\mathcal A$ be a line not passing through $P$. Once again, from formula $(\star)$ at the beginning of this section, $D_H+2T_H=s-1$, where $s=8$ or $9$. Also, since $P$ is modular and $P\notin H$, $D_H+T_H=4$, the total number of the singular points on $H$. So $T_H=s-5$.

If $s=9$, then $T_H=4$, and so $D_H=0$. Again, by \cite[Corollary 3.5]{AnTo}, such line arrangement is not realizable over the fields of complex numbers.

So $s=8$, and therefore $T_H=3$ and $D_H=1$. If $L\in\mathcal A$ with $P\in L$, then $D_L+2T_L+3=7$. This leads to $D_L=0, T_L=2$, or $D_L=2, T_L=1$, or $D_L=4, T_L=0$ (we will see that this last option is not possible). Suppose $H_1,H_2,H_3,H_4$ are the lines of $\A$ not passing through $P$. Suppose $Q_1, Q_2, Q_3, Q_4$ are the four singular points of $H_1$; i.e., $\{Q_i\}=H_1\cap L_i$, where $L_1,L_2,L_3,L_4$ are the lines of $\A$ passing through modular point $P$. Suppose $m_{Q_1}=2$. Then $m_{Q_2}=m_{Q_3}=m_{Q_4}=3$; suppose $H_1\cap H_j\cap L_j=\{Q_j\}, j=2,3,4$.

Suppose we constructed $H_2$. Since $\mathcal A$ is supersolvable, the intersection (i.e., singular) point $H_3\cap H_2$ is either the point $H_2\cap L_1$ or $H_2\cap L_4$. Suppose it is former, $Q:=H_3\cap H_2=H_2\cap L_1$. This leaves only one option to construct $H_4$; it is the line connecting the intersection points $L_2\cap H_3$, $Q_4$, and $L_3\cap H_2$. This is exactly the line arrangement of type (Vc).
\end{proof}

\section{Computational approach}

Let $\mathcal A=\{V(l_1),\ldots,V(l_s)\}\subset \mathbb P^2$ be a rank 3 line arrangement. Suppose we fix the equations of the defining linear forms: $l_i=a_ix+b_iy+c_iz\in S:=\mathbb K[x,y,z]$ for each $i\in \{1,\ldots,s\}$. Let $\theta=P\frac{\partial}{\partial x}+Q\frac{\partial}{\partial y}+R\frac{\partial}{\partial z}$ be a degree $d$ logarithmic derivation of $\mathcal{A}$, that is, $P,Q,R \in S_d.$

For practical purpose we can suppose $l_1=x,~l_2=y,~l_3=z$, then we can rewrite $$\theta=P'x\frac{\partial}{\partial x}+Q'y\frac{\partial}{\partial y}+R'z\frac{\partial}{\partial z}~ \mathrm{with}~ P',Q',R' \in S_{d-1}.$$

Using the Euler derivation, we can write $\theta=P'\theta_E+(Q'-P')y\frac{\partial}{\partial y}+(R'-P')z\frac{\partial}{\partial z}.$ Therefore, we can suppose $P'=0$, and in the following sections we will work with a logarithmic derivation with the form $$\theta=Q'y\frac{\partial}{\partial y}+R'z\frac{\partial}{\partial z}~ \mathrm{with}~ Q',R' \in S_{d-1}.$$

\begin{rem}\label{divisonlines}
Since $l_1=x,~l_2=y,~l_3=z$, for each $i\geq 4$ we have two of $a_i,b_i,c_i$ being nonzero. Reordering, if necessary, we can group the lines $l_1,\ldots,l_s$ in four cases:

\begin{itemize}
	\item[i)] $a_4,b_4,\ldots,a_{\mathfrak{j}},b_{\mathfrak{j}} \neq 0$ and $c_4,\ldots,c_{\mathfrak{j}} = 0;$
	\item[ii)] $a_{\mathfrak{j}+1},c_{\mathfrak{j}+1},\ldots,a_{\mathfrak{l}},c_{\mathfrak{l}} \neq 0$ and $b_{\mathfrak{j}+1},\ldots,b_{\mathfrak{l}} = 0;$
	\item[iii)] $b_{\mathfrak{l}+1},c_{\mathfrak{l}+1},\ldots,b_{\mathfrak{n}},c_{\mathfrak{n}} \neq 0$ and $a_{\mathfrak{l}+1},\ldots,a_{\mathfrak{n}} = 0;$
	\item[iv)] $a_{\mathfrak{n}+1},b_{\mathfrak{n}+1},c_{\mathfrak{n}+1},\ldots,a_{s},b_{s},c_{s} \neq 0,$ with $4\leq \mathfrak{j}< \mathfrak{l}< \mathfrak{n}< s.$
\end{itemize}
For our purposes we also can suppose $b_j=1,$ for $j\in \{4,\ldots,\mathfrak{j} \}$ in case i), and $c_i=1$ for for cases ii), iii), and iv) (i.e., for $i\in\{\mathfrak{j}+1,\ldots,s\}$).
\end{rem}

The main goal is to analyze the shape of such a logarithmic derivation, towards obtaining conditions for a line arrangement to possess a minimal logarithmic derivation to degrees 2 and 3 through a matrix argument. In what follows we will use ``$\Leftrightarrow$'' to denote equivalent matrices, i.e., matrices obtained from one-another via elementary row operations.

\subsection{Degree 2 logarithmic derivations}	
	
First we review the beginning of \cite[Section 2.3]{To} while having a more organized approach. Let $\theta=Q'y\frac{\partial}{\partial y}+R'z\frac{\partial}{\partial z},~ \mathrm{with}~ Q',R' \in S_1,$ be a degree 2 logarithmic derivation of $\mathcal{A}$. For each $i\in \{1,\ldots,s\}$ we have $\theta(l_i)=l_iT_i$ for some $T_i\in S_1$. Let $Q'=q_1x+q_2y+q_3z,~R'=r_1x+r_2y+r_3z,~T_i=t^i_1x+t^i_2y+t^i_3z\in S_1.$ Therefore,
$$b_iyQ'+c_izR'=(a_ix+b_iy+c_iz)T_i.$$

Comparing coefficients, the following six relations are obtained

\begin{eqnarray}\label{EQdegree2}
\displaystyle \Big\{\begin{array}{ccc}
0=a_it^i_1\hspace{1.5cm} & \hspace{0.5cm}b_iq_2=b_it^i_2\hspace{1.2cm} &\hspace{0.5cm} c_ir_3=c_it^i_3\hspace{2.1cm}\\
b_iq_1=a_it^i_2+b_it^i_1& \hspace{0.5cm}c_ir_1=a_it^i_3+c_it^i_1&\hspace{0.5cm} b_iq_3+c_ir_2=b_it^i_3+c_it^i_2
\end{array}
\end{eqnarray}

\begin{rem}\label{comput-quadratic} If we projectively eliminate the parameters $t_1^i,t_2^i,t_3^i$ in (\ref{EQdegree2}), we obtain that $\mathcal A$ has a quadratic logarithmic derivation that is not a multiple of $\theta_E$ if and only if the points dual to the lines of $\mathcal A$ are contained in the variety $$V(xy(yq_1-xq_2), xz(zr_1-xr_3), yz[y(q_3-r_3)-z(q_2-r_2)]),$$ for some constants $q_i,r_j$, not all zero. In a snapshot, this is \cite[Corollary 2.2]{To}, which is used to prove, via some challenging computations, the classification of line arrangements with quadratic minimal logarithmic derivations.

As we mentioned in Section 2.1, in this project our approach has been simplified, and it goes on the opposite direction of \cite{To}: with the powerful tools provided by the addition-deletion results in \cite{TaDiSti}, we first find the classification, and as we will see below, we find the corresponding minimal quadratic logarithmic derivations (i.e., the constants $q_i,r_j$).
\end{rem}

\bigskip

Applying Remark \ref{divisonlines} to (\ref{EQdegree2}) and eliminating the parameters $t_1^i,t_2^i,t_3^i$, we have the following equations for $q_1,q_2,q_3,r_1,r_2,r_3$ and their associated matrices  $\mathcal{N}_{\mathrm{i}},\mathcal{N}_{\mathrm{ii}},\mathcal{N}_{\mathrm{iii}},\mathcal{N}_{\mathrm{iv}}$ respectively.

$\bullet$ \underline{Situation i)}: for $4\leq j \leq \mathfrak{j}$ we have $q_3=0$ and $q_1-a_jq_2=0$ with matrix
$$\mathcal{N}_{\mathrm{i}}=\left[\begin{array}{cccccc}
	0&0&1&0&0&0\\
	1&-a_4&0&0&0&0\\
	\vdots&\vdots&\vdots&\vdots&\vdots&\vdots\\ 	
	1&-a_j&0&0&0&0\\
	\vdots&\vdots&\vdots&\vdots&\vdots&\vdots\\
	1&-a_{\mathfrak{j}}&0&0&0&0\\
\end{array}\right].$$

$\bullet$ \underline{Situation ii)}: for $\mathfrak{j}+1\leq l \leq \mathfrak{l}$ we have $r_2=0$ and $r_1-a_lr_3=0$ with matrix
$$\mathcal{N}_{\mathrm{ii}}=\left[\begin{array}{cccccc}
0&0&0&0&1&0\\
0&0&0&1&0&-a_{\mathfrak{j}+1}\\
\vdots&\vdots&\vdots&\vdots&\vdots&\vdots\\ 	
0&0&0&1&0&-a_l\\
\vdots&\vdots&\vdots&\vdots&\vdots&\vdots\\
0&0&0&1&0&-a_{\mathfrak{l}}\\
\end{array}\right].$$

$\bullet$ \underline{Situation iii)}: for $\mathfrak{l}+1\leq n \leq \mathfrak{n}$ we have $q_1-r_1=0$ and $(q_2-r_2)-b_n(q_3-r_3)=0$ with matrix
$$\mathcal{N}_{\mathrm{iii}}=\left[\begin{array}{cccccc}
1&0&0&-1&0&0\\
0&1&-b_{\mathfrak{l}+1}&0&-1&b_{\mathfrak{l}+1}\\
\vdots&\vdots&\vdots&\vdots&\vdots&\vdots\\ 	
0&1&-b_n&0&-1&b_n\\
\vdots&\vdots&\vdots&\vdots&\vdots&\vdots\\
0&1&-b_{\mathfrak{n}}&0&-1&b_{\mathfrak{n}}\\
\end{array}\right].$$

$\bullet$ \underline{Situation iv)}: for $\mathfrak{n}+1\leq i \leq s$ we have $b_iq_1-a_iq_2=0$, $r_1-a_ir_3=0$, and $(q_2-r_2)-b_i(q_3-r_3)=0$ with matrix $\mathcal{N}_{\mathrm{iv}}$ obtained by concatenating the following matrices:
\begin{tiny} $$\iota=\left[\begin{array}{cccccc}
b_{\mathfrak{n}+1}&-a_{\mathfrak{n}+1}&0&0&0&0\\
\vdots&\vdots&\vdots&\vdots&\vdots&\vdots\\ 	
b_i&-a_i&0&0&0&0\\
\vdots&\vdots&\vdots&\vdots&\vdots&\vdots\\
b_s&-a_s&0&0&0&0\\
\end{array}\right]
\varsigma=\left[\begin{array}{cccccc}
0&0&0&1&0&-a_{\mathfrak{n}+1}\\
\vdots&\vdots&\vdots&\vdots&\vdots&\vdots\\ 	
0&0&0&1&0&-a_i\\
\vdots&\vdots&\vdots&\vdots&\vdots&\vdots\\
0&0&0&1&0&-a_s\\
\end{array}\right]
\mathfrak{i}=\left[\begin{array}{cccccc}
0&-1&b_{\mathfrak{n}+1}&0&1&-b_{\mathfrak{n}+1}\\
\vdots&\vdots&\vdots&\vdots&\vdots&\vdots\\ 	
0&-1&b_i&0&1&-b_i\\
\vdots&\vdots&\vdots&\vdots&\vdots&\vdots\\
0&-1&b_s&0&1&-b_s\\
\end{array}\right].$$
\end{tiny}

We can represent all the equations above by using the following null product of representative matrix:

$$\mathcal{M}_{(3s-2\mathfrak{n})\times 6}\cdot \left[\begin{array}{cccccc}
q_1&q_2&q_3&r_1&r_2&r_3
\end{array}\right]_{6\times 1}^{t}=0,$$
where the matrix $\mathcal{M}:=\mathcal{M}_{(3s-2\mathfrak{n})\times 6}$ is the concatenation of the matrices $\mathcal{N}_{\mathrm{i}},\mathcal{N}_{\mathrm{ii}},\mathcal{N}_{\mathrm{iii}},\mathcal{N}_{\mathrm{iv}}$. This linear system has non trivial solution if and only if $\mathrm{rank}(\mathcal{M})<\min\{3s-2\mathfrak{n},6\}.$

Using the classification of all line arrangements in $\mathbb{P}^2$ with $r(\mathcal{A})=2$ by (\ref{classification}) we can determine the shape of each degree 2 logarithmic derivation associated to each the defining polynomial.

\begin{cor} \label{quadratic-log}
The degree 2 logarithmic derivation associated to each the defining polynomial obtained in (\ref{classification}) is, respectively,
\begin{itemize}
	\item [(I)] If $s\geq 5,$ then $\theta=(x+y)(y\frac{\partial}{\partial y}+z\frac{\partial}{\partial z}).$ \\
               If $s=4$ then $\theta=\alpha(x+y)(y\frac{\partial}{\partial y}+z\frac{\partial}{\partial z})+\beta(t_4y+z)z\frac{\partial}{\partial z}$, for $\alpha,\beta \in \mathbb K$, not both equal to zero.
	\item [(II)] $\theta=(x+y+z)(y\frac{\partial}{\partial y} +  z\frac{\partial}{\partial z}).$
	\item [(III)] $\theta=(x+y+2z)y\frac{\partial}{\partial y} + (x+z) z\frac{\partial}{\partial z}$.
\end{itemize} 		
\end{cor}

\begin{proof}
\begin{itemize}
		\item [(I)] For this type the associated matrix $\mathcal{M}$ has the following shape:

		$$\mathcal{M}_\mathrm{I}= \left[\begin{array}{cccccc}
    	0&0&1&0&0&0\\
		1&-1&0&0&0&0\\
        1&0&0&-1&0&0\\
		0&1&-t_4&0&-1&t_4\\
		\vdots&\vdots&\vdots&\vdots&\vdots&\vdots\\ 	
		0&1&-t_j&0&-1&t_j\\
		\vdots&\vdots&\vdots&\vdots&\vdots&\vdots\\
		0&1&-t_s&0&-1&t_s\\
		\end{array}\right]_{s\times 6}\Longleftrightarrow \left[\begin{array}{cccccc}
		0&0&1&0&0&0\\
		1&-1&0&0&0&0\\
		1&0&0&-1&0&0\\
		0&1&0&0&-1&t_4\\
        0&0&0&0&0&t_5-t_4\\
		\vdots&\vdots&\vdots&\vdots&\vdots&\vdots\\ 	
		0&0&0&0&0&t_j-t_4\\
		\vdots&\vdots&\vdots&\vdots&\vdots&\vdots\\
		0&0&0&0&0&t_s-t_4\\
		\end{array}\right]_{s\times 6}.$$ This leads to the corresponding shape of the logarithmic derivation.

		For $s\geq 5$, since $t_j-t_4 \neq 0 ~\forall~j$, we have: \hspace{3cm}For $s=4$, we have:
		
		$\mathcal{M}_\mathrm{I}= \left[\begin{array}{cccccc}
		0&0&1&0&0&0\\
		1&-1&0&0&0&0\\
		1&0&0&-1&0&0\\
		0&1&0&0&-1&0\\
		0&0&0&0&0&1\\
		0&0&0&0&0&0\\
		\vdots&\vdots&\vdots&\vdots&\vdots&\vdots\\
		0&0&0&0&0&0\\
		\end{array}\right]_{s\times 6}.\hspace*{1.5cm}\mathcal{M}_\mathrm{I}= \left[\begin{array}{cccccc}
		0&0&1&0&0&0\\
		1&-1&0&0&0&0\\
		1&0&0&-1&0&0\\
		0&1&0&0&-1&t_4\\
		\end{array}\right]_{5\times 6}.$ This leads to the corresponding shape of the logarithmic derivation.
		
		\item [(II)] For this type the associated matrix $\mathcal{M}$ has the following shape:
		
		$$\mathcal{M}_\mathrm{II}= \left[\begin{array}{cccccc}
        1&0&0&-1&0&0\\
        0&1&-t_4&0&-1&t_4\\
        \vdots&\vdots&\vdots&\vdots&\vdots&\vdots\\ 	
        0&1&-t_j&0&-1&t_j\\
        \vdots&\vdots&\vdots&\vdots&\vdots&\vdots\\
        0&1&-t_s&0&-1&t_s\\
		1&-1&0&0&0&0\\
		0&0&0&1&0&-1\\
		0&-1&1&0&1&-1\\
		\end{array}\right]_{(s+1)\times 6}\Longleftrightarrow  \left[\begin{array}{cccccc}
		1&0&0&-1&0&0\\
		0&1&-t_4&0&-1&t_4\\
     	0&0&t_4-t_5&0&0&t_5-t_4\\
		\vdots&\vdots&\vdots&\vdots&\vdots&\vdots\\ 	
		0&0&t_4-t_j&0&0&t_j-t_4\\
		\vdots&\vdots&\vdots&\vdots&\vdots&\vdots\\
		0&0&t_4-t_s&0&0&t_s-t_4\\
		1&-1&0&0&0&0\\
		0&0&0&1&0&-1\\
		0&-1&1&0&1&-1\\
		\end{array}\right]_{(s+1)\times 6}.$$

		For $s\geq 5$, since $t_j-t_4 \neq 0$ for all $j$, we have $\mathcal{M}_\mathrm{II}= \left[\begin{array}{cccccc}
		1&0&0&-1&0&0\\
		0&1&0&0&-1&0\\
		0&0&-1&0&0&1\\
		1&-1&0&0&0&0\\
		0&0&0&1&0&-1\\
		0&0&0&0&0&0\\
		\vdots&\vdots&\vdots&\vdots&\vdots&\vdots\\
		0&0&0&0&0&0\\
		\end{array}\right]_{(s+1)\times 6}.$

		For $s=4$:\\ $\mathcal{M}_\mathrm{II}= \left[\begin{array}{cccccc}
		1&0&0&-1&0&0\\
		0&1&-t_4&0&-1&t_4\\
		1&-1&0&0&0&0\\
		0&0&0&1&0&-1\\
		0&-1&1&0&1&-1\\
		\end{array}\right]_{5\times 6} \Longleftrightarrow  \left[\begin{array}{cccccc}
		1&0&0&-1&0&0\\
		0&1&-t_4&0&-1&t_4\\
		1&-1&0&0&0&0\\
		0&0&0&1&0&-1\\
		0&0&1-t_4&0&0&t_4-1\\
		\end{array}\right]_{5\times 6}.$ Since $t_4\neq 1$, we obtain the corresponding shape of the logarithmic derivation.

		\item [(III)] For this type the associated matrix $\mathcal{M}$ has the following shape:

				$$\mathcal{M}_\mathrm{III}= \left[\begin{array}{cccccc}
		0&0&0&0&1&0\\
		0&0&0&1&0&-1\\
		1&0&0&-1&0&0\\
		0&1&-1&0&-1&1\\
		1&-1&0&0&0&0\\
		0&0&0&1&0&-1\\
		0&-1&1&0&1&-1\\
		\end{array}\right]_{7\times 6}\Longleftrightarrow\left[\begin{array}{cccccc}
		0&0&0&0&1&0\\
		0&0&0&1&0&-1\\
		1&0&0&-1&0&0\\
		0&1&-1&0&0&1\\
		1&-1&0&0&0&0\\
		0&0&0&0&0&0\\
		0&0&0&0&0 &0\\
		\end{array}\right]_{7\times 6}.$$ This leads to the corresponding shape of the logarithmic derivation.
	\end{itemize}
	
\end{proof}

\subsection{Degree 3 logarithmic derivations}

In this section we reproduce the same argument as in the previous section, but for degree 3 logarithmic derivations. Let $\theta=Q'y\frac{\partial}{\partial y}+R'z\frac{\partial}{\partial z}~ \mathrm{with}~ Q',R' \in S_2$ be a degree 3 logarithmic derivation of $\mathcal{A}$. For each $i\in \{1,\ldots,s\}$ we have $\theta(l_i)=l_iT_i$ for some $T_i\in S_2$.

Let $Q'=q_1x^2+q_2y^2+q_3z^2+q_4xy+q_5xz+q_6yz,R'=r_1x^2+r_2y^2+r_3z^2+r_4xy+r_5xz+r_6yz,~T_i=t^i_1x^2+t^i_2y^2+t^i_3z^2+t^i_4xy+t^i_5xz+t^i_6yz \in k[x,y,z]_2.$ Therefore,
$$b_iyQ'+c_izR'=(a_ix+b_iy+c_iz)T_i.$$

Comparing coefficients, the following ten relations are obtained
\begin{eqnarray}\label{EQdegree3}
\displaystyle \Biggl\{\begin{array}{ccc}
&&\\
0=a_it^i_1     & \hspace{1.1cm}b_iq_4=a_it^i_2+b_it^i_4 & b_iq_3+c_ir_6=b_it^i_3+c_it^i_6 \\
b_iq_2=b_it^i_2& \hspace{1.1cm}c_ir_5=a_it^i_3+c_it^i_5 & b_iq_6+c_ir_2=b_it^i_6+c_it^i_2 \\
c_ir_3=c_it^i_3& \hspace{1.1cm}b_iq_1=a_it^i_4+b_it^i_1 & \hspace{1.1cm}b_iq_5+c_ir_4=a_it^i_6+b_it^i_5+c_it^i_4 \\
               & \hspace{1.1cm}c_ir_1=a_it^i_5+c_it^i_1 &
\end{array}
\end{eqnarray}

If we projectively eliminate $t_1^i,t_2^i,t_3^i$ in (\ref{EQdegree3}), we obtain the following result which is similar to Remark \ref{comput-quadratic}.

\begin{lem}\label{dual-points} Let $\mathcal A\subset\mathbb P^2$ be a line arrangement with defining linear forms $l_1=x$, $l_2=y$, $l_3=z$, and $l_i=a_ix+b_iy+c_iz, i\geq 4$. Let $l_1^{\vee}:=[1,0,0], l_2^{\vee}:=[0,1,0], l_3^{\vee}:=[0,0,1], l_i^{\vee}:=[a_i,b_i,c_i], i\geq 4$ be the points dual to the lines of $\mathcal A$. Let $\mathcal A^{\vee}:=\{l_1^{\vee},\ldots,l_s^{\vee}\}$.

Then, $\mathcal A$ has a cubic logarithmic derivation that is not a multiple of the Euler derivation if and only if $\mathcal A^{\vee}$ is included in $V(F_1,F_2,F_3,F_4)$, where
\begin{eqnarray}
F_1&=& xy(q_2x^2-q_4xy+q_1y^2)\nonumber\\
F_2&=& xz(r_3x^2-r_5xz+r_1z^2)\nonumber\\
F_3&=& yz[(q_3-r_3)y^2-(q_6-r_6)yz+(q_2-r_2)z^2]\nonumber\\
F_4&=&xyz[(q_3-2r_3)xy^2-(q_2-r_6)xyz-(q_5-r_5)y^2z+(q_4-r_4)yz^2],\nonumber
\end{eqnarray} where $q_i,r_j\in\mathbb K$ are not all zero, and they will be the coefficients of $Q'$ and $R'$ shown above.
\end{lem}

\begin{exm} \label{example2} Let us consider the following two examples that are presented at the beginning of \cite[Section 4]{To1}.\footnote{They are examples denoted there with $\mathcal A_2$, and $\mathcal A_6$; in the later we performed a change of variables $x+y+z\leftrightarrow x, y\leftrightarrow y, z\leftrightarrow z$, so it will fit into the statement of Lemma \ref{dual-points}.}
$$\mathcal B_1:=V(xyz(x-z)(x-2z)(y-z)(y-2z))$$
$$\mathcal B_2:=V(xyz(x-2y-2z)(2x+2y-z)(2x-y+2z)(3x-12y-4z)).$$ $\mathcal B_1$ has two minimal cubic logarithmic derivations (actually it is free with exponents $(1,3,3)$), whereas $\mathcal B_2$ has six minimal degree five logarithmic derivations (generating the first syzygy module of the Jacobian ideal), so no cubic logarithmic derivation. But the minimal graded free resolutions of the ideals of $\mathcal B_1^{\vee}$ and of $\mathcal B_2^{\vee}$, respectively, are the same:

$$0\longrightarrow S^2(-5)\longrightarrow S(-2)\oplus S^2(-4).$$

In Example \ref{example1}, though there are no cubic logarithmic derivations other than multiples of $\theta_E$, it is worth mentioning that $I(\mathcal A_1^{\vee})$ and $I(\mathcal A_2^{\vee})$ are both ideal complete intersection of two cubic forms. To have the same minimal graded free resolutions is expected to happen, since the two arrangements have the ``same pictures'' (same combinatorics). Is it possible to find an example of two line arrangements having the same combinatorics, but different graded minimal free resolutions of the ideals of the points dual to the lines?
\end{exm}

\bigskip

Applying Remark \ref{divisonlines} to (\ref{EQdegree3}) and eliminating the parameters $t_1^i,t_2^i,t_3^i$, we have the following equations for $q_1,\ldots,q_6,r_1,\ldots,r_6$ and their associated matrices  $\mathcal{N}_{\mathrm{i}},\mathcal{N}_{\mathrm{ii}},\mathcal{N}_{\mathrm{iii}},\mathcal{N}_{\mathrm{iv}}$ respectively.

$\bullet$ \underline{Situation i)}: for $4\leq j \leq \mathfrak{j}$ we have $q_3=0,~q_1+a_j^2q_2-a_jq_4=0$ and $q_5-a_jq_6=0$ with matrix

$$\mathcal{N}_{\mathrm{i}}=\left[\begin{array}{cccccccccccc}
0&0&1&0&0&0&0&0&0&0&0&0\\
1&a_4^2&0&-a_4&0&0&0&0&0&0&0&0\\
\vdots&\vdots&\vdots&\vdots&\vdots&\vdots&\vdots&\vdots&\vdots&\vdots&\vdots&\vdots\\ 	
1&a_j^2&0&-a_j&0&0&0&0&0&0&0&0\\
\vdots&\vdots&\vdots&\vdots&\vdots&\vdots&\vdots&\vdots&\vdots&\vdots&\vdots&\vdots\\
1&a_{\mathfrak{j}}^2&0&-a_{\mathfrak{j}}&0&0&0&0&0&0&0&0\\
0&0&0&0&1&-a_4&0&0&0&0&0&0\\
\vdots&\vdots&\vdots&\vdots&\vdots&\vdots&\vdots&\vdots&\vdots&\vdots&\vdots&\vdots\\
0&0&0&0&1&-a_j&0&0&0&0&0&0\\
\vdots&\vdots&\vdots&\vdots&\vdots&\vdots&\vdots&\vdots&\vdots&\vdots&\vdots&\vdots\\
0&0&0&0&1&-a_{\mathfrak{j}}&0&0&0&0&0&0
\end{array}\right].$$

$\bullet$ \underline{Situation ii)}: for $\mathfrak{j}+1\leq l \leq \mathfrak{l}$ we have $r_2=0,~r_1+a_l^2r_3-a_lr_5=0$ and $r_4-a_lr_6=0$  with matrix
$$\mathcal{N}_{\mathrm{ii}}=\left[\begin{array}{cccccccccccc}
0&0&0&0&0& 0&0&1&0&0&0&0\\
0&0&0&0&0&0&1& 0&a_{\mathfrak{j}+1}^2&0&-a_{\mathfrak{j}+1}&0\\
\vdots&\vdots&\vdots&\vdots&\vdots&\vdots&\vdots&\vdots&\vdots&\vdots&\vdots&\vdots\\ 	
0&0&0&0&0&0&1&0&a_l^2&0&-a_l&0\\
\vdots&\vdots&\vdots&\vdots&\vdots&\vdots&\vdots&\vdots&\vdots&\vdots&\vdots\\ 	
0&0&0&0&0&0&1&0&a_{\mathfrak{l}}^2&0&-a_{\mathfrak{l}}&0\\
0&0&0&0&0&0&0&0&0&1&0&-a_{\mathfrak{j}+1}\\
\vdots&\vdots&\vdots&\vdots&\vdots&\vdots&\vdots&\vdots&\vdots&\vdots&\vdots&\vdots\\ 	
0&0&0&0&0&0&0&0&0&1&0&-a_l\\
\vdots&\vdots&\vdots&\vdots&\vdots&\vdots&\vdots&\vdots&\vdots&\vdots&\vdots&\vdots\\ 	
0&0&0&0&0&0&0&0&0&1&0&-a_{\mathfrak{l}}\\
\end{array}\right].$$

$\bullet$ \underline{Situation iii)}: for $\mathfrak{l}+1\leq n \leq \mathfrak{n}$ we have $q_1-r_1=0,~q_2-r_2+b_n^2(q_3-r_3)-b_n(q_6-r_6)=0$ and $(q_4-r_4)-b_n(q_5-r_5)=0$ and  with matrix
$$\mathcal{N}_{\mathrm{iii}}=\left[\begin{array}{cccccccccccc}
1&0&0&0&0&0&-1&0&0& 0&0&0\\
0&1&b_{\mathfrak{l}+1}^2&0&0&-b_{\mathfrak{l}+1}&0&-1&-b_{\mathfrak{l}+1}^2& 0&0&b_{\mathfrak{l}+1}\\
\vdots&\vdots&\vdots&\vdots&\vdots&\vdots&\vdots&\vdots&\vdots&\vdots&\vdots&\vdots\\ 	
0&1&b_{n}^2&0&0&-b_{n}&0&-1&-b_{n}^2& 0&0&b_{n}\\
\vdots&\vdots&\vdots&\vdots&\vdots&\vdots&\vdots&\vdots&\vdots&\vdots&\vdots&\vdots\\ 	
0&1&b_{\mathfrak{n}}^2&0&0&-b_{\mathfrak{n}}&0&-1&-b_{\mathfrak{n}}^2& 0&0&b_{\mathfrak{n}}\\
0&0&0&1&-b_{\mathfrak{l}+1}&0&0&0&0&-1&b_{\mathfrak{l}+1}&0\\
\vdots&\vdots&\vdots&\vdots&\vdots&\vdots&\vdots&\vdots&\vdots&\vdots&\vdots&\vdots\\ 	
0&0&0&1&-b_n&0&0&0&0&-1&b_n&0\\
\vdots&\vdots&\vdots&\vdots&\vdots&\vdots&\vdots&\vdots&\vdots&\vdots&\vdots&\vdots\\ 	
0&0&0&1&-b_{\mathfrak{n}}&0&0&0&0&-1&b_{\mathfrak{n}}&0
\end{array}\right].$$

$\bullet$ \underline{Situation iv)}: for $\mathfrak{n}+1\leq i \leq s$ we have $b_i^2q_1+a_i^2q_2-a_ib_iq_4=0,~r_1+a_i^2r_3-a_ir_5=0,~q_2-b_i^2r_3-b_i(q_6-r_6)=0$ and $-a_iq_2-2a_ib_i^2r_3+a_ib_ir_6+b_i(q_4-r_4)-b_i^2(q_5-r_5)=0$ with matrix $\mathcal{N}_{\mathrm{iv}}$ obtained by concatenating the following matrices:
$$\mathcal{T}=\left[\begin{array}{cccccccccccc}
b_{\mathfrak{n}+1}^2&a_{\mathfrak{n}+1}^2&0&-a_{\mathfrak{n}+1}b_{\mathfrak{n}+1}&0&\cdots&0\\
\vdots&\vdots&\vdots&\vdots&\vdots&&\vdots\\ 	
b_i^2&a_i^2&0&-a_ib_i&0&\cdots&0\\
\vdots&\vdots&\vdots&\vdots&\vdots&&\vdots\\
b_{s}^2&a_{s}^2&0&-a_{s}b_{s}&0&\cdots&0\\
\end{array}\right],~\mathcal{U}=\left[\begin{array}{cccccccccccc}
0&\cdots&0&1& 0&a_{\mathfrak{n}+1}^2&0&-a_{\mathfrak{n}+1}&0\\
\vdots&&\vdots&\vdots&\vdots&\vdots&\vdots\\ 	
0&\cdots&0&1&0&a_i^2&0&-a_i&0\\
\vdots&&\vdots&\vdots&\vdots&\vdots&\vdots\\ 	
0&\cdots&0&1&0&a_{s}^2&0&-a_{s}&0\\
\end{array}\right],$$

$$\mathcal{V}=\left[\begin{array}{cccccccccccc}
0&1&0&0&0&-b_{\mathfrak{n}+1}&0&0&-b_{\mathfrak{n}+1}^2& 0&0&b_{\mathfrak{n}+1}\\
\vdots&\vdots&\vdots&\vdots&\vdots&\vdots&\vdots&\vdots&\vdots&\vdots&\vdots&\vdots\\ 	
0&1&0&0&0&-b_{i}&0&0&-b_{i}^2& 0&0&b_{i}\\
\vdots&\vdots&\vdots&\vdots&\vdots&\vdots&\vdots&\vdots&\vdots&\vdots&\vdots&\vdots\\ 	
0&1&0&0&0&-b_{s}&0&0&-b_{s}^2& 0&0&b_{s}\\
\end{array}\right],$$

$$  \mathcal{W}= \left[\begin{array}{cccccccccccc}
0&-a_{\mathfrak{n}+1}&0&b_{\mathfrak{n}+1}&-b_{\mathfrak{n}+1}^2&0&0&0&-2a_{\mathfrak{n}+1}b_{\mathfrak{n}+1}^2&-b_{\mathfrak{n}+1}&b_{\mathfrak{n}+1}^2&a_{\mathfrak{n}+1}b_{\mathfrak{n}+1}\\
\vdots&\vdots&\vdots&\vdots&\vdots&\vdots&\vdots&\vdots&\vdots&\vdots&\vdots&\vdots\\ 	
0&-a_i&0&b_i&-b_i^2&0&0&0&-2a_ib_i^2&-b_i&b_i^2&a_ib_i\\
\vdots&\vdots&\vdots&\vdots&\vdots&\vdots&\vdots&\vdots&\vdots&\vdots&\vdots&\vdots\\ 	
0&-a_s&0&b_s&-b_s^2&0&0&0&-2a_sb_s^2&-b_s&b_s^2&a_sb_s\\
\end{array}\right].$$

\vspace{0.3cm}

We can represent all the equations above by using the following null product of representative matrix:

$$\mathcal{M}_{(3(s-1)-\mathfrak{n})\times 12}\cdot \left[\begin{array}{llllllllllll}
q_1&q_2&q_3&q_4&q_5&q_6&r_1&r_2&r_3&r_4&r_5&r_6
\end{array}\right]_{12\times 1}^{t}=0,$$
where the matrix $\mathcal{M}:=\mathcal{M}_{(3(s-1)-\mathfrak{n})\times 12}$ is the concatenation of the matrices $\mathcal{N}_{\mathrm{i}},\mathcal{N}_{\mathrm{ii}},\mathcal{N}_{\mathrm{iii}},\mathcal{N}_{\mathrm{iv}}$. This linear system has non trivial solution if and only if $\mathrm{rank}(\mathcal{M})<\min\{3(s-1)-\mathfrak{n},12\}.$

Using the classification of all line arrangements in $\mathbb{P}^2$ with $r(\mathcal{A})=3$ by (\ref{cubic}) we can determine the shape of each degree 3 logarithmic derivation associated to each the defining polynomial.

This comes handy especially when dealing with cases (IIIa)-(Vb), because $s$ has specific values. For example, the associated matrix $\mathcal{M}_{\mathrm{Vb}}$  to the type (Vb) $F=xyz(x+y)(x+z)(-x+z)(y+z)(-y+z)$ (here we changed $x\mapsto -x$) is

{\small$$\left[\begin{array}{cccccccccccc}
0&0&1&0&0&0&0&0&0&0&0&0\\
1&1&0&-1&0&0&0&0&0&0&0&0\\
0&0&0&0&1&-1&0&0&0&0&0&0\\
0&0&0&0&0& 0&0&1&0&0&0&0\\
0&0&0&0&0&0&1& 0&1&0&-1&0\\
0&0&0&0&0&0&1&0&1&0&1&0\\
0&0&0&0&0&0&0&0&0&1&0&-1\\
0&0&0&0&0&0&0&0&0&1&0&1\\
1&0&0&0&0&0&-1&0&0& 0&0&0\\
0&1&1&0&0&-1&0&-1&-1& 0&0&1\\
0&1&1&0&0&1&0&-1&-1& 0&0&-1\\
0&0&0&1&1&0&0&0&0&1&-1&0\\
0&0&0&1&-1&0&0&0&0&1&1&0\\
\end{array}\right]\Longleftrightarrow\left[\begin{array}{cccccccccccc}
0&0&1&0&0&0&0&0&0&0&0&0\\
1&1&0&0&0&0&0&0&0&0&0&0\\
0&0&0&1&0&0&0&0&0&0&0&0\\
0&0&0&0&1&0&0&0&0&0&0&0\\
0&0&0&0&0&1&0&0&0&0&0&0\\
1&0&0&0&0&0&-1&0&0&0&0&0\\
0&0&0&0&0&0&0&1&0&0&0&0\\
0&0&0&0&0&0&1&0&1&0&0&0\\
0&1&0&0&0&0&0&0&-1&0&0&0\\
0&0&0&0&0&0&0&0&0&1&0&0\\
0&0&0&0&0&0&0&0&0&0&1&0\\
0&0&0&0&0&0&0&0&0&0&0&1\\
0&0&0&0&0&0&0&0&0&0&0&0\\
\end{array}\right].$$}

Therefore, the shape of the logarithmic derivation of this type of arrangement is $$\theta=(-x^2+y^2)y\frac{\partial}{\partial y} + (-x^2+z^2)z\frac{\partial}{\partial z}.$$

\subsubsection{Non recursive cubic logarithmic derivations.}

For types (Ia) - (IIIc), the corresponding line arrangements $\mathcal A$ are obtained by adding an extra line $H=V(l)$ to line arrangements $\mathcal B$ with minimal quadratic logarithmic derivation. The minimal cubic logarithmic derivations will be obtained by simply multiplying by $l$ the logarithmic derivations obtained in Corollary \ref{quadratic-log}, except when $\mathcal B$ has also a minimal cubic logarithmic derivation.

We are interested in this exceptional case. So it may happen that $$\theta = l'\rho_2+\rho_3,$$ where $l'\in S_1$ and $\rho_2$ is a minimal quadratic logarithmic derivation, and $\rho_3$ is a minimal cubic logarithmic derivation, both of $\mathcal B$. Also, since $r(\mathcal A)=3$, then $\rho_2(l)\notin \langle l\rangle$.

From Corollary \ref{quadratic-log}, we know the shape of $\rho_2$. So we are left to finding how $\rho_3$ looks like. Now we can use the shape of the defining polynomials of $\mathcal B$ as expressed in Theorem \ref{classification}, together with the matrix approach we have been discussing in this section.

\medskip

$\bullet$ If $\mathcal B$ is of type (I), then it is supersolvable, and we want it to have exponents $(1,2,3)$. Then its defining polynomial is $xyz(x+y)(t_4y+z)(t_5y+z)$. After some change of variables, we can assume this polynomial is $xyz(x+y)(y+z)(ty+z)$, where $t\neq 0,1$. Modulo linear multiples of $\rho_2$, we have $\displaystyle\rho_3=(ty+z)(y+z)z\frac{\partial}{\partial_z}$.

After a change of coordinates, we can assume that the defining polynomial of $\mathcal A$ in Theorem \ref{cubic} is $F=xyz(x+y)(y+z)(ty+z)l$, where $l=bx+y$ for (Ia), $l=bx+z$ for (Ib) and (Ic), $l=ax+by+z$ for (Id).

The question is if we can find $l'=\alpha x+\beta y +\gamma z$ such that $l'\rho_2(l)+\rho_3(l)\in\langle l\rangle$.

\begin{itemize}
  \item[(Ia)] If $l=bx+y$, then we can take $l'=l$.
  \item[(Ib)] If $l=bx+z$, $b\neq -1,$ then, in order to have such $\theta$, $b=-t$. But with this, we get the picture of (Ic), since we have three triple points, instead of just two: $\{x,y,x+y\},\{x,z,-tx+z\},\{x+y,-tx+z,ty+z\}$.
  \item[(Ic)] If $l=-x+z$, then we can take $l'=x-ty-2z$.
  \item[(Id)] If $l=ax+by+z$, then $b=t$ or $b=1$. But in this case we get two triple points instead of just one: $\{x,y,x+y\},\{x,ax+by+z,by+z\}$, so not a type (Id) arrangement.
\end{itemize}

\medskip

$\bullet$ If $\mathcal B$ is of type (II), then, after some simple calculation with Macaulay 2 (\cite{EiGrSt}), its defining polynomial is $xyz(x+y+z)(ty+z), t\neq 0,1$. In this case $\mathcal B$ has two minimal cubic logarithmic derivations: $$\rho_3=(ty+z)(x+y+z)(\alpha y\frac{\partial}{\partial y}+\beta z\frac{\partial}{\partial z}).$$ We also know that $\rho_2=(x+y+z)(y\frac{\partial}{\partial y}+z\frac{\partial}{\partial z})$.

If $l=ax+by+cz$, then we need to have $$l'(x+y+z)(by+cz)+(ty+z)(x+y+z)(\alpha by+\beta cz)=(ax+by+cz)Q, Q\in S_2.$$ Since $\gcd(l,x+y+z)=1$, then $Q=(x+y+z)P$, where $P\in S_1$.

\begin{itemize}
  \item[(IIb)] If $c=0$ and $b=1$, then we have $$l'y+(ty+z)\alpha y=(ax+y)P,$$ leading to $P=\delta y$. Then $\alpha=0$ and $l'=\delta(ax+y)$.
  \item[(IIa)] If $c=1$, and $a$ and $b$ are general enough ($a,b\neq 0$, $b\neq t$), we have $$l'(by+z)+(ty+z)(\alpha by+\beta z)=(ax+by+cz)P.$$ After some calculations, $\beta=\alpha$, and we can choose $l'=ax+(b-\alpha t)y+(1-\alpha)z$ (in this instance $P=by+z$).
\end{itemize}

\medskip

$\bullet$ If $\mathcal B$ is of type (III), as we mention before we can obtain all the minimal cubic logarithmic derivations by matrix computations.

\section{Appendix: graphic arrangements}

Let $G=(V,E)$ be a simple (no loops, no multiple edges) undirected graph with $V=\{1,\ldots,n\}$, $n\geq 2$. The {\em graphic arrangement corresponding to $G$} is $\mathcal A(G):=\{V(x_i-x_j)|(i,j)\in E\}$. The rank of $\mathcal A(G)$ equals $n$ minus the number of connected components of $G$.

Consider the derivation $$\theta=x_1^2\frac{\partial}{\partial x_1}+\cdots+x_n^2\frac{\partial}{\partial x_n}.$$ Obviously, $\theta(x_i-x_j)=x_i^2-x_j^2=(x_i-x_j)(x_i+x_j)$, so $\theta$ is a logarithmic derivation of $\mathcal A(G)$, not a multiple of the Euler derivation. So $$r(G):=r(\mathcal A(G))\leq 2.$$ We are interested when $r(G)=1$.

Below, $\delta(G)$ denotes the minimum degree of a vertex of $G$. Suppose $G$ is connected. $S\subset V$ is called {\em vertex-cut set}, if $G-S$ is disconnected. The smallest cardinality of a vertex-cut set is called {\em vertex-connectivity} of $G$, and it is denoted $k(G)$. By convention, $k(K_n)=n-1$.

\begin{flushleft}
$\begin{array}{lc}
\hspace{0.25cm}\textrm{A vertex}~v\in V~\textrm{such that}~ G-\{v\}~\textrm{is disconnected}~(\textrm{hence}~k(G)=1)~\textrm{is} \\
\textrm{called \emph{articulated vertex}, and}~G~\textrm{with}~k(G)=1~\textrm{is called {\emph {articulated graph}}}.\\
\\
\end{array}
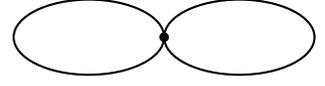
\hspace*{0.3cm}
\begin{tikzpicture}
\draw[thick](0, 0) ellipse (1 and 0.5) (2, 0) ellipse (1 and 0.5);
\draw [thick,fill] (1, 0) circle (0.05);
\end{tikzpicture}$
\end{flushleft}

\begin{lem}\label{bowtie} Let $G=(V,E)$ be a simple undirected graph with $|E|\geq 2$ and $\delta(G)\geq 1$. If $G$ is disconnected, or if $G$ is articulated, then $r(G)=1$.
\end{lem}
\begin{proof} To prove the lemma, we will appeal to the discussion in \cite[Section 2.1]{To}, where it says that a hyperplane arrangement $\mathcal A$ has $r(\mathcal A)=1$ if and only if $\mathcal A$ is reducible (i.e., $\mathcal A=\mathcal A_1\times\mathcal A_2$).

If $G$ is disconnected with components $G_1$ and $G_2$, then the defining polynomial of $\mathcal A(G)$ is the product of the defining polynomials of $\mathcal A(G_1)$ and $\mathcal A(G_2)$, and each of these is expressed in different sets of variables. Hence $\mathcal A(G)$ is reducible.

\medskip

Suppose $G$ is articulated at a vertex $v$, with the two induced (connected) subgraphs $G_1$ and $G_2$, such that $G=G_1\cup G_2$ and $G_1\cap G_2=\{v\}$.

Let $T$ be a spanning tree for $G$. Then we can make a change of variables such that for each $(i,j)$ edge of $T$, $x_i-x_j$ gets assigned a new variable. Since $T$ is obtained by removing efficiently edges from $G$ to ``destroy'' cycles, then all forms $x_a-x_b$ corresponding to other edges $(a,b)$ of $G$ will be linear combinations of these new variables.

But $T$ consists of a spanning tree of $G_1$ and a spanning tree of $G_2$ glued at the vertex $v$. Using the change of variables associated to spanning trees as above, the claim follows (there is no cycle of $G$ consists of edges from both $G_1$ and $G_2$).
\end{proof}

The converse is the following. But first, we summarize \cite[Proposition 2.87]{OrTe}: Let $e$ be an edge of $G$, and let $G':=G-e$ (the deletion of the edge $e$), and let $G'':= G/e$ (the contraction w.r.t. the edge $e$). If $H\in \mathcal A:=\mathcal A(G)$ is the hyperplane corresponding to $e$, then $$\mathcal A':=\mathcal A\setminus\{H\} = \mathcal A(G') \mbox{ and }\mathcal A'':=\mathcal A^H=\mathcal A(G'').$$

\begin{lem} Let $G=(V,E)$ be a simple connected graph with $s:=|E|\geq 2$. If $r(G)=1$, then $G$ is articulated.
\end{lem}
\begin{proof} We will prove the result by induction on $s\geq 2$. Of course, if $s=2$, then $G$ is a path of length 2, hence it is articulated.

Suppose $s\geq 3$. If $G$ has a leaf, then the neighbor of that leaf is an articulated vertex of $G$. Suppose $\delta(G)\geq 2$. Let $e\in E$, and consider $G'=G-e$. Then, $\delta(G')\geq 1$. By Theorem \ref{additiondeletion}, since the rank of $\mathcal A(G)$ is at least 2, we have $r(G')=1$.

$G'$ is disconnected or, by induction, $G'$ is articulated; i.e, $G'=G_1'\cup G_2'$, with $G_1'\cap G_2'=\emptyset$, or $G_1'\cap G_2'=\{v\}$.

In the first case, since $G$ is connected, then $e$ must be a bridge connecting $G_1'$ and $G_2'$. But then, any of the end vertices of $e$ is an articulated vertex for $G$.

In the second case, suppose $e$ connects the two halves $G_1'$ and $G_2'$ at vertices $w_1$ and $w_2$ respectively, and it is not incident to $v$ (for any other placement of $e$, $G$ is articulated at $v$).

If $G$ is the triangle $vw_1w_2$, then $G=K_3$, so $r(G)=2$. Contradiction. Otherwise, if one looks at $G'':=G/e$, we have: \begin{itemize}
	\item $G''$ is connected.
	\item $\mathrm{rank}(\mathcal{A}(G''))\geq 2.$
	\item $G''$ has less than or equal $s-1$ edges.
	\end{itemize}

If $G''$ is articulated, then $G$ is articulated. If $G''$ is not articulated, by induction, $r(G'')=2$. But then, by Theorem \ref{additiondeletion} (2), $r(G)=r(G')+1=2$; a contradiction.
\end{proof}

We end by mentioning that since $r(G)\leq 2$ for any graphic arrangement, it becomes more appealing to analyze ``the maximal degree of a Jacobian relation'' (so a maximal degree of a generating syzygy of the Jacobian ideal), and \cite[Corrolary 4.6]{Wa} finds an interesting lower bound of this invariant, in terms of the maximum number of new triangles that can be formed by adding an edge to $G$.

\vskip 0.3in

\noindent {\bf Acknowledgment.} Ricardo Burity thanks the Department of Mathematics of the University of Idaho (USA) for the hospitality during his stay. He is grateful as well to CAPES (MEC, Brazil) for funding the current Post-Doctoral scholarship at the University of Idaho.

We are very grateful to Wayne Ng Kwing King and Jean Vall\`{e}s for pointing to us the missing case (Vc) in Theorem \ref{cubic}. This type is derived from the remaining case $n_4=1$ in the proof of this theorem, that we originally did not analyze.

\bigskip

\renewcommand{\baselinestretch}{1.0}
\small\normalsize 

\bibliographystyle{amsalpha}

\begin{thebibliography}{10}

\bibitem{Ab} T. Abe,
        {\em Double points of free projective line arrangements},
        Int. Math. Res. Not., to appear; a version is available at arXiv: 1911.10754.

\bibitem{AbDi} T. Abe and A. Dimca,
    {\em On complex supersolvable line arrangements},
    J. Algebra {\bf 552} (2020), 38--51.

\bibitem{TaDiSti}  T. Abe, A. Dimca and G. Sticlaru,
	{\em Addition–deletion results for the minimal degree of logarithmic derivations of hyperplane arrangements and maximal Tjurina line arrangements}, J. Algebraic Combin., to appear; a version is available at arXiv: 1908.06885.

\bibitem{AnTo} B. Anzis and \c{S}. Toh\v{a}neanu,
   {\em On the geometry of real and complex supersolvable line arrangements},
   J. Combin. Theory, Ser. A {\bf 140} (2016), 76--96.

\bibitem{Di} A. Dimca,
    {\em Curve arrangements, pencils, and Jacobian syzygies},
    Michigan Math. J. {\bf 66} (2017), 347--365.

\bibitem{DiSti} A. Dimca and G. Sticlaru,
            {\em Line and rational curve arrangements, and Walther's inequality},
            Atti Accad. Naz. Lincei Rend. Lincei Mat. Appl., to appear, a version is available at arXiv: 1803.05386.

\bibitem{DiSti2} A. Dimca and G. Sticlaru,
            {\em On the exponents of free and nearly free projective plane curves},
            Rev. Mat. Complut. {\bf 30} (2017), 259--268.

\bibitem{duWa} A.A. du Plessis and C.T.C. Wall,
            {\em Application of the theory of the discriminant to highly singular plane curves},
            Math. Proc. Cambridge Phil. Soc. {\bf 126} (1999), 259--266.

\bibitem{EiGrSt} D. Eisenbud, D. Grayson and M. Stillman,
			{Macaulay2, a software system for research in algebraic geometry},
        Available at http://www.math.uiuc.edu/Macaulay2/.

\bibitem{GeHaMi} A.V. Geramita, B. Harbourne and J. Migliore, {\em Star configurations in $\mathbb{P}^n$}, J. Algebra {\bf 376} (2013), 279--299.

\bibitem{HaHa} K. Hanumanthu and B. Harbourne,
        {\em Real and complex supersolvable line arrangements in the projective plane},
        J. Algebraic Combin., to appear; a version is available at arXiv: 1907.07712.

\bibitem{OrTe}  P. Orlik and H. Terao,
            Arrangements of Hyperplanes,
            Grundlehren Math. Wiss., Bd. 300, Springer-Verlag, Berlin-Heidelberg-New York, 1992.

\bibitem{To0} \c{S}. Toh\v{a}neanu,
	{\em On freeness of divisors on $\mathbb P^2$},
	Comm. Algebra {\bf 41} (2013), 2916--2932.

\bibitem{To} \c{S}. Toh\v{a}neanu,
	{\em Projective duality of arrangements with quadratic logarithmic vector fields},
	Discrete Math. {\bf 339} (2016), 54--61.

\bibitem{To1} \c{S}. Toh\v{a}neanu,
    {\em Free arrangements with low exponents}, a version is available at arXiv: 1707.07091.

\bibitem{Wa} M. Wakefield,
        {\em Derivation degree sequences of non-free arrangements}, Comm. Algebra {\bf 47} (2019), 2654--2666.

\bibitem{Yu} S. Yuzvinsky,
        {\em First two obstructions to the freeness of arrangements},
        Trans. Amer. Math. Soc. {\bf 335} (1993), 231--244.

\bibitem{Zi} G. Ziegler,
        {\em Combinatorial construction of logarithmic differential forms}, Adv. Math. {\bf 76} (1989), 116--154.
	
\end{thebibliography}

\end{document}